\documentclass[a4paper,12pt]{article}
\usepackage{amssymb}
\usepackage{amsmath}
\usepackage{amsfonts}
\usepackage{fontsmpl}
\usepackage{accents}

\newtheorem{theorem}{Theorem}

\newtheorem{algorithm}[theorem]{Algorithm}

\newtheorem{corollary}[theorem]{Corollary}

\newtheorem{definition}[theorem]{Definition}

\newtheorem{lemma}[theorem]{Lemma}

\newtheorem{proposition}[theorem]{Proposition}

\newenvironment{proof}[1][Proof]{\noindent\textbf{#1.} }{\ \rule{0.5em}{0.5em}}
\input{tcilatex}
\begin{document}

\title{\textsl{Construction of Nilpotent Jordan Algebras Over any Arbitrary
Fields}}
\author{A.Hegazi , H.Abdelwahab \\
%EndAName
Mathematics department, Faculty of Science, Mansoura University, Egypt\\
hegazi@mans.edu.eg\\
hanii\_i@yahoo.com}
\date{}
\maketitle

\begin{abstract}
A cohomological approach for classifying nilpotent Jordan algebras, every
nilpotent Jordan algebras can be constructed by the second cohomolgy of
nilpotent Jordan algebras of low dimension. Also we describe the analogue of 
\textit{Skjelbred-Sund method} for classifying nilpotent lie algebras to
classify nilpotent Jordan algebras and carry out a procudure for
constructing nilpotent Jordan algebras over arbitrary fields. By using this
method, we construct all non-isomorphic nilpotent Jordan algebras up to
dimension three over any field, and all non-isomorphic nilpotent Jordan
algebras of dimension four over an algebraic closed field of characteristic $%
\neq 2$ and over the real field $%
%TCIMACRO{\U{211d} }%
%BeginExpansion
\mathbb{R}
%EndExpansion
$. Also commutative nilpotent associative algebras are classified, since any
commutative nilpotent associative algebras are nilpotent Jordan algebras. We
show that there are up to isomorphism $13$ nilpotent Jordan algebras of
dimension $4$ over an algebraic closed field of characteristic $\neq 2,$ and 
$4$ of those are not associative, yielding $9$\ commutative nilpotent
associative algebras. Also up to isomorphism there are $17$ nilpotent Jordan
algebras of dimension $4$ over the real field $%
%TCIMACRO{\U{211d} }%
%BeginExpansion
\mathbb{R}
%EndExpansion
$, and $5$ of those are not associative, yielding $12$\ commutative
nilpotent associative algebras.
\end{abstract}



\section{Introduction}

Jordan algebras were introduced in the early 1930's by a physicist, P.
Jordan, in an attempt to generalize the formalism of quantum mechanics.
Little appears to have resulted in this direction, but un-anticipated
relationships between these algebras and Lie groups and the foundations of
geometry have been discovered.

The structure of this paper, in section $2$ we define the notion of
Extension of Jordan algebra which is similar to the concept of Extension of
Groups, let $N,F$ and $G$ be groups and let $G$ have a normal subgroup $\bar{%
N}$ which is isomorphic to $N$, that is, $\bar{N}$ $\cong N$. Recall that a
subgroup is called normal or invariant if one has $g\bar{N}g^{-1}=\bar{N}$
for all $g$ $\in $ $G$. (At the Jordan algebras level a normal subgroup
yields an ideal.) The group $G$ is called an extension of $F$ by $N$ if the
factor group $G/N$ is isomorphic with the group $F$, i.e. $G/\bar{N}$ $\cong
F$. The relationship between the groups $N,F$ and $G$ can be represented by
the sequence%
\begin{equation*}
N\overset{\Gamma }{\longrightarrow }G\overset{\Lambda }{\longrightarrow }F
\end{equation*}%
where $\Gamma $ is an injective group homomorphism with im$\Gamma =\bar{N}$
and where $\Lambda $ is a surjective group homomorphism with ker$\Lambda =%
\bar{N}$. A systematic study of group extensions has been performed by the
German mathematician O. Schreier in $1926$.

In section $3,$ we focus on the Centeral Extension and in sections $4,5$ and 
$6$ we descripe the analouge of of \textit{The Skjelbred-Sund Method, }and
carry out a procudure for constructing nilpotent Jordan algebras over any
arbitrary field. Finally, classify nilpotent Jordan algebras of dimension $%
\leq 4.$

\begin{definition}
A Jordan algebra $J$ is a vector space over a field $K$ equipped with a
symmetric bilinear map $B:J\times J\longrightarrow J$ such that :%
\begin{equation*}
B(B(x,x),B(x,y)=B(x,B(B(x,x),y))\text{ },\text{all }x,y\in J.
\end{equation*}%
which is called the Jordan identity.
\end{definition}

Let $J$ be a Jordan algebra over a field $K.$For any integer $m\in 
%TCIMACRO{\U{2115} }%
%BeginExpansion
\mathbb{N}
%EndExpansion
$ we define a lower centeral series as the descending chain of ideals%
\begin{equation*}
c^{1}(J)=J\supseteq c^{2}(J)=B(J,J)\supseteq c^{3}(J)=B(c^{2}(J),J)\supseteq
.....\supseteq c^{m}(J)=B(c^{m-1}(J),J)
\end{equation*}

\begin{definition}
A Jordan algebra $J$ over a field $K$ is said to be nilpotent of nilindex $n$
if :%
\begin{equation*}
c^{n}(J)=\{0\}\text{ \ \ and \ \ }c^{n-1}(J)\neq \{0\}.
\end{equation*}%
Also we call $n$ the nilpotency class of $J$ and $J$ is called $n$-step
nilpotent.
\end{definition}

\begin{definition}
Let $J$ be a Jordan algebra over a field $K$ then the ideal : 
\begin{equation*}
Z(J)=\{x\in J:B(x,y)=0\ \forall x\in J\}
\end{equation*}%
is called the centre of $J.$
\end{definition}

So, any nilpotent Jordan algebra have non-trivial centre $Z(J)\neq \{0\}.$In
the Jordan algebra $J,$ we will denote the product $B(x,y)$ by $x\circ y$
and $x^{2}$ for $B(x,x).$Then the jordan identity is written as follow :%
\begin{equation*}
x^{2}\circ (x\circ y)=x\circ (x^{2}\circ y)\text{ \ , all }x,y\in J.
\end{equation*}%
\textsl{\ }

\section{Extensions of Jordan algebras}

In this section we are going to focus on extensions of Jordan algebras.
Loosely speaking an extension of a Jordan algebra $J$ is an enlargement of $%
J $ by some other Jordan algebra. To be somewhat more concrete we consider
the following construction. Starting with two Jordan algebras $L$ and $M$
over the same field $K$ we consider the Cartesian product $J=L\times M$.
Elements of this set are ordered pairs $(l,m)$ with $l\in L$ and $m\in M$.
Defining addition of such pairs by $(l,m)+(l^{^{\prime }},m^{^{\prime }}):=$ 
$(l+l^{^{\prime }},m+m^{^{\prime }})$ and multiplication by scalars $\alpha
\in K$ as $\alpha (l,m):=$ $(\alpha l,\alpha m)$ the set $J$ \ becomes a
vector space which will also be denoted by $J$. Using the Jordan
multiplication on $L$ and $M$ we define on $J$ \ the multiplication%
\begin{equation*}
(l,m)\circ (l^{^{\prime }},m^{^{\prime }})=(l\circ _{L}l^{^{\prime }},m\circ
_{M}m^{^{\prime }})
\end{equation*}%
One easily sees that this is a Jordan multipliction on $J$. Defining next
the map $\epsilon $ $:m\in M\longrightarrow $ $(0,m)$ $\in J$ and the map $%
\lambda :$ $(l,m)\in J\longrightarrow $ $m\in M$ one readily verifies that $%
\epsilon $ is an injective Jordan algebra homomorphism, while $\lambda $ is
a surjective Jordan algebra homomorphism. Moreover, im $\epsilon $ $=$ ker $%
\lambda $. The Jordan algebra $J$ with these properties is called a trivial
extension of of $L$ by $M$ ( or of $M$ by $L$ ) .

Instead of denoting the elements of $L$ by ordered pairs we will frequently
use the notation $(l,m)=l+m$ . The vector space $J$ is then written as%
\begin{equation*}
J=L\oplus M
\end{equation*}%
and the Jordan multiplication is in this notation given by%
\begin{equation*}
(l+m)\circ (l^{^{\prime }}+m^{^{\prime }})=l\circ _{L}l^{^{\prime }}+m\circ
_{M}m^{^{\prime }}
\end{equation*}%
This example of a trivial extension is generalized in the following concept.
Let $L,J$ and $M$ be Jordan algebras and let these algebras be related in
the following way.

\begin{itemize}
\item There exists an injective Jordan algebra homomorphism%
\begin{equation*}
\epsilon :L\longrightarrow J\text{ .}
\end{equation*}

\item There exists an surjective Jordan algebra homomorphism%
\begin{equation*}
\lambda :J\longrightarrow M\text{ .}
\end{equation*}

\item The Jordan algebra homomorphisms $\epsilon $ and $\lambda $ are
related by%
\begin{equation*}
\text{im }\epsilon =\text{ker }\lambda \text{ }.
\end{equation*}
\end{itemize}

Then the Jordan algebra $J$ \ is called an extension of $M$ by $L$. The
relationship is summarized by the sequence%
\begin{equation*}
L\overset{\epsilon }{\longrightarrow }J\overset{\lambda }{\longrightarrow }M
\end{equation*}%
We elaborate a little on this concept. Since $\lambda $ is a jordan algebra
homomorphism one obtains that ker $\lambda $ is an ideal in $J$ and since $%
\lambda $ is surjective First Isomorphism Theorem entails a Jordan algebra
isomorphism between the quotient algebra $J/$ker $\lambda $ and $L$ ,%
\begin{equation*}
J/\text{ker }\lambda \cong M
\end{equation*}%
Using im $\epsilon =$ ker $\lambda $ this relation can be written as%
\begin{equation*}
J/\text{im }\epsilon \cong M
\end{equation*}%
Since $\epsilon $ is injective the Jordan algebras $L$ and im $\epsilon $
are Jordan isomorphic, i.e $J/L\cong M$. From these properties one sees that
it makes sense to call the Jordan algebra $J$ is an extension of $M$ by $L$.

\begin{definition}
Let $L$, $J$ and $M$ be Jordan algebras over the field $K$. Let 
\begin{equation*}
\lambda :J\longrightarrow M
\end{equation*}
be a surjective Jordan algebra homomorphism and 
\begin{equation*}
\epsilon :L\longrightarrow J
\end{equation*}%
an injective Jordan algebra homomorphism. Then the sequence%
\begin{equation}
L\overset{\epsilon }{\longrightarrow }J\overset{\lambda }{\longrightarrow }M
\label{h}
\end{equation}%
is called an \textbf{extension }of $M$ by $L$ if $\epsilon $ maps $L$ onto
the kernel ker $\lambda $ $\subset $ $J$ \ of the map $\lambda $. That is if%
\begin{equation*}
\text{im }\epsilon =\text{ker }\lambda \text{ }.
\end{equation*}%
The kernel ker $\lambda $ is called the \textbf{kernel of the extension}.
\end{definition}

Loosely speaking $J$ \ itself, instead of the sequence (\ref{h}), will be
called an extension of $M$ by $L$.

It may happen that there exist several extensions of $M$ by $L$. To classify
extensions we define the notion of \textbf{equivalent extensions}.

\begin{definition}
Two sequences%
\begin{equation*}
L\overset{\epsilon }{\longrightarrow }J\overset{\lambda }{\longrightarrow }M
\end{equation*}%
and%
\begin{equation*}
L\overset{\epsilon ^{^{\prime }}}{\longrightarrow }J^{^{\prime }}\overset{%
\lambda ^{^{\prime }}}{\longrightarrow }M
\end{equation*}%
are called \textbf{equivalent extensions} if there exists a Jordan algebra
isomorphism $\phi :J\longrightarrow J^{^{\prime }}$ such that%
\begin{equation*}
\phi \circ \epsilon =\epsilon ^{^{\prime }},\text{ \ }\lambda ^{^{\prime
}}\circ \phi =\lambda \text{ .}
\end{equation*}
\end{definition}

One easily sees that\textbf{\ equivalence of extensions} is an equivalence
relation.

The concept of a Jordan algebra extension can be formulated more succinctly
using the concept of an exact sequence.Let $\{J_{i}\}$ be a sequence of
Jordan algebras and $\{\phi _{i}\}$ a sequence of Jordan algebra
homomorphisms $\phi _{i}:$ $J_{i}\longrightarrow J_{i+1}$, then the sequence%
\begin{equation*}
.....\longrightarrow J_{i-1}\overset{\phi _{i-1}}{\longrightarrow }J_{i}%
\overset{\phi _{i}}{\longrightarrow }J_{i+1}\overset{\phi _{i+1}}{%
\longrightarrow }.....
\end{equation*}%
is called \textbf{exact} if one has for each $i$%
\begin{equation*}
\text{im }\phi _{i-1}=\text{ker }\phi _{i}\text{ .}
\end{equation*}%
Let $L,J$ and $M$ be Jordan algebras and let us denote by $0$ the Jordan
algebra consisting of the zero element only. Then the sequence $0\overset{f}{%
\longrightarrow }L\overset{g}{\longrightarrow }J$ is exact if and only if $g$
is an injective Jordan algebra homomorphism. Indeed, since im $f=0$ $\in L$
the requirement ker $g$ $=$ im $f$ \ forces $g$ to be injective. One usually
omits the mapping $f$ i.e., one writes $0\longrightarrow L\overset{g}{%
\longrightarrow }J$ . Likewise the sequence $J\overset{f}{\longrightarrow }%
M\longrightarrow 0$ is exact if and only if $f$ \ is a surjective Jordan
algebra homomorphism. Consequently we have the following proposition.

\begin{proposition}
The sequence $L\longrightarrow J\longrightarrow M$ of Jordan algebras is an 
\textbf{extension} of $M$ by $L$ if and only if the sequence%
\begin{equation*}
0\longrightarrow L\overset{\epsilon }{\longrightarrow }J\overset{\lambda }{%
\longrightarrow }M\longrightarrow 0
\end{equation*}%
is \textbf{exact .}
\end{proposition}

Next we introduce some particular types of extensions. For this we need the
following concept from linear algebra. Two subspaces $V_{1}$ and $V_{2}$ of
a vector space $V$ are called complementary if $V$ is the direct sum of $%
V_{1}$ and $V_{2}$.

\begin{definition}
An extension%
\begin{equation*}
L\overset{\epsilon }{\longrightarrow }J\overset{\lambda }{\longrightarrow }M
\end{equation*}

is called:

\textbf{trivial} if there exists an ideal $I$ $\subset J$ complementary to
ker $\lambda $, i.e.%
\begin{equation*}
J=ker\text{ }\lambda \oplus I\text{ \ \ \ \ \ }(\text{ Jordan algebra direct
sum }),
\end{equation*}

\textbf{split }if there exists a Jordan subalgebra $S\subset J$
complementary to ker $\lambda $,

i.e.%
\begin{equation*}
J=ker\text{ }\lambda \oplus S\text{\ \ \ \ \ }(\text{ vector space direct
sum }),
\end{equation*}

\textbf{centeral }if the kernel ker $\lambda $ is contained in the center $%
Z(J)$ of $J$, i.e.%
\begin{equation*}
ker\text{ }\lambda \subset Z(J)\text{ .}
\end{equation*}
\end{definition}

\section{2-Cocycles on Jordan algebra}

In this section we take a closer look at central extensions of Jordan
algebras. The Jordan algebras in this section will be over the field $K$ .
Recall that the sequence of Jordan algebras%
\begin{equation*}
L\overset{\epsilon }{\longrightarrow }J\overset{\lambda }{\longrightarrow }M
\end{equation*}%
is a central extension (of $M$ by $L$) if one has for the Jordan algebra
homomorphisms $\epsilon $ and $\lambda $ the following properties: $\epsilon 
$\ is injective, $\lambda $ is surjective, im $\epsilon =$ ker $\lambda $
and the kernel ker $\lambda $ is contained in the center $Z(J)$ of $J$.
Hence im $\epsilon $ = ker $\lambda $ is a Jordan algebra with trivial
multiplicatin. Since $Z(J)$ is a Jordan subalgebra with trivial
multiplicatin and since $L$ and im $\epsilon $\ are isomorphic, $L$ is a
Jordan algebra with trivial multiplicatin too. A central extension of a
Jordan algebra $M$ by a Jordan algebra with trivial multiplicatin( a vector
space ) $L$ can be obtained with the help of a so called $2$-cocycle on $M$.

\begin{definition}
\label{hh}Let $J$ be a Jordan algebra and $V$ be avector space over $K$. A
bilinear map%
\begin{equation*}
\theta :J\times J\longrightarrow V
\end{equation*}%
is called a Jordan $2$-cocycle from $J$ to $V$\ if it satisfies for all $%
x,y\in J$ the following conditions :
\end{definition}

\begin{itemize}
\item $\theta (x,y)=\theta (y,x)$ \ \ \ \ ($\theta $ is symmetric ) .

\item $\theta (x^{2},x\circ y)=\theta (x,x^{2}\circ y)$ \ \ \ \ ( Jordan
identity for $2$-cocycles ) .
\end{itemize}

The set of $2$-cocycles from $J$ to $V$ is denoted by $Z^{2}(J,V)$ . One
easily sees that $Z^{2}(J,V)$ is a vector space if one defines the vector
space operations as follows. Let $\theta _{1}$ and $\theta _{2}$ be $2$%
-cocycles rom $J$ to $V$, then their linear combination $\lambda _{1}$ $%
\theta _{1}+$ $\lambda _{2}$ $\theta _{2}$ ( $\lambda _{1},\lambda _{2}\in K$
) is defined by%
\begin{equation*}
(\lambda _{1}\theta _{1}+\lambda _{2}\theta _{2})(x,y):=\lambda _{1}\theta
_{1}(x,y)+\lambda _{2}\theta _{2}(x,y)\text{ .}
\end{equation*}%
This linear combination is again a Jordan $2$-cocycle.

Let $\theta \in $ $Z^{2}(J,V)$ ,and set $J_{\theta }=J\oplus V.$ Using
Jordan multiplication on $J$ and the cocycle $\theta $ from $J$ to $V$ e
define on $J_{\theta }$ the multiplication 
\begin{equation*}
(x+v)\circ (y+w)=x\circ _{J}y+\theta (x,y)
\end{equation*}%
For $x,y\in J,v,w\in V$ .

\begin{lemma}
\ $J_{\theta }$ is a Jordan algebra if and only if $\ \theta \in $ $%
Z^{2}(J,V)$ .
\end{lemma}

\begin{proof}
For $x+v,y+w\in J_{\theta },$we have that :%
\begin{eqnarray*}
(x+v)^{2}\circ ((x+v)\circ (y+w)) &=&(x\circ _{J}x+\theta (x,x))\circ
(x\circ _{J}y+\theta (x,y)) \\
&=&x^{2}\circ _{J}(x\circ _{J}y)+\theta (x^{2},x\circ _{J}y) \\
(x+v)\circ ((x+v)^{2}\circ (y+w)) &=&(x+v)\circ ((x^{2}+\theta (x,x))\circ
(y+w)) \\
&=&(x+v)\circ (x^{2}\circ _{J}y+\theta (x\circ _{J}x,y)) \\
&=&x\circ _{J}(x^{2}\circ _{J}y)+\theta (x,x^{2}\circ _{J}y)
\end{eqnarray*}%
From Jordan identity it follows that $(x+v)^{2}\circ ((x+v)\circ
(y+w))=(x+v)\circ ((x+v)^{2}\circ (y+w)$ if and only if $\theta
(x^{2},x\circ y)=\theta (x,x^{2}\circ y)$ then $J_{\theta }$ is a Jordan
algebra if and only if $\theta \in $ $Z^{2}(J,V)$ .
\end{proof}

\begin{lemma}
Let $\theta \in $ $Z^{2}(J,V)$ then $J_{\theta }$ is a centeral extension of 
$J$ by $V$ .
\end{lemma}

\begin{proof}
From the previous lemma \ $J_{\theta }$ is a Jordan algebra and $V$ is a
Jordan algebra with the trivial multiplication . The sequence $V\overset{i}{%
\longrightarrow }J\oplus V\overset{\pi }{\longrightarrow }J$ \ is exact( $i$
be the injection map and $\pi $ be the projection map ) and $V=$ker $\pi
\subset Z(J_{\theta })$ . Hence $J_{\theta }$ is a centeral extension of $J$
by $V$ .
\end{proof}

Hence, a Jordan $2$-cocycle $\theta $\ from $J$ to $V$ ( dim$V=k$ ) allows
for the construction of a $k$-dimensional central extension $J_{\theta
}=J\oplus V$ where the Jordan multiplication on $J_{\theta }$ is given by%
\begin{equation*}
(x+v)\circ (y+w)=x\circ _{J}y+\theta (x,y)
\end{equation*}%
For $x,y\in J,v,w\in V$ .

Now we will show that a Jordan $2$-cocycles can be obtained from
the"abstract form" of a central extension. Considering the extension%
\begin{equation*}
L\overset{\epsilon }{\longrightarrow }J\overset{\lambda }{\longrightarrow }M
\end{equation*}%
with im $\epsilon =$ker $\lambda \subset Z(J)$ we show that one can find $2$%
-cocycles from $M$ to $L$. Consider a linear map $s:M\longrightarrow J$
satisfying%
\begin{equation*}
\lambda \circ s=id_{M}
\end{equation*}%
A map with this property is called a section of $J$. With the help of a
section one can define a bilinear map $\theta ^{^{\prime }}:M\times
M\longrightarrow J$ by taking for all $x,y\in M$%
\begin{equation}
\theta ^{^{\prime }}(x,y)=s(x\circ y)-s(x)\circ s(y)  \label{hhh}
\end{equation}%
Notice that $\theta ^{^{\prime }}$ is identically zero if $s$ is a Jordan
algebra homomorphism. Notice also that $\theta ^{^{\prime }}$ is symmetric.
From equation $\lambda \circ s=id_{M}$ and fact that $\lambda $ is a Jordan
algebra homomorphism one sees that%
\begin{equation*}
\lambda (\theta ^{^{\prime }}(x,y))=0
\end{equation*}%
Hence, we have for all $x,y$ $\in M$%
\begin{equation}
\theta ^{^{\prime }}(x,y)\in \text{ker }\lambda \subset Z(J)\text{ .}
\label{hhhh}
\end{equation}%
Using the injectivity of the map $\epsilon :L\longrightarrow J$ to define
the map%
\begin{equation*}
\theta :M\times M\longrightarrow J\text{ .}
\end{equation*}%
Given by%
\begin{equation*}
\theta :=\epsilon ^{-1}\circ \theta ^{^{\prime }}.
\end{equation*}%
$\theta $ is bilinear and symmetric . It remains to show that $\theta $
satisfy the second condition in Defenition \ref{hh} . Let $x,y\in M$ then by
using Jordan identity, equation (\ref{hhh}) and property (\ref{hhhh}) we
have 
\begin{eqnarray*}
\theta ^{^{\prime }}(x^{2},x\circ y) &=&s(x^{2}\circ (x\circ
y))-s(x^{2})\circ s(x\circ y) \\
&=&s(x\circ (x^{2}\circ y))-s(x^{2})\circ (\theta ^{^{\prime
}}(x,y)+s(x)\circ s(y)) \\
&=&s(x\circ (x^{2}\circ y))-s(x^{2})\circ (s(x)\circ s(y)) \\
&=&s(x\circ (x^{2}\circ y))-s(x)\circ (s(x^{2})\circ s(y)) \\
&=&s(x\circ (x^{2}\circ y))-s(x)\circ (\theta ^{^{\prime
}}(x^{2},y)+s(x^{2}\circ y)) \\
&=&s(x\circ (x^{2}\circ y))--s(x)\circ s(x^{2}\circ y) \\
&=&\theta ^{^{\prime }}(x,x^{2}\circ y)\text{ .}
\end{eqnarray*}%
Hence, $\theta (x^{2},x\circ y)=\epsilon ^{-1}(\theta ^{^{\prime
}}(x^{2},x\circ y))=\epsilon ^{-1}(\theta ^{^{\prime }}(x,x^{2}\circ
y))=\theta (x,x^{2}\circ y)$ . Then $\theta $ is a $2$-cocycles from $M$ to $%
L$. We have thus shown that $2$-cocycles are quite natural objects in
central extensions.

\begin{corollary}
$J_{\theta }$ is a centeral extension of $J$ by $V$ if and only if $\ \theta
\in $ $Z^{2}(J,V)$.
\end{corollary}

We now proceed with the properties of cocycles. A special type of $2$%
-cocycles is given by so-called $2$-coboundaries.

\begin{definition}
Let $J$ be a Jordan algebra and $V$ be avector space over $K$. Then a linear
map%
\begin{equation*}
f:J\longrightarrow V.
\end{equation*}%
is called a $1$-cochain from $J$ to $V$. The set of all $1$-cochains from $J$
to $V$ is denoted by $C^{1}(J,V).$
\end{definition}

Notice that $C^{1}(J,V)$ be the vector space $Hom(J,V)$ and hence it has,
just as$Z^{2}(J,V)$, an obvious vector space structure. Using cochains one
defines $2$-coboundaries as follows.

\begin{definition}
Let $J$ be a Jordan algebra,$V$ be avector space over $K$ and let $%
f:J\longrightarrow V$ be a $1$-cochain from $J$ to $V$. Then the bilinear map%
\begin{equation*}
\delta f:J\times J\longrightarrow V.
\end{equation*}%
defined by 
\begin{equation*}
(\delta f)(x,y):=f(x\circ y).
\end{equation*}%
is called a $2-$coboundary from $J$ to $V$.
\end{definition}

Next we show that $\delta f$ is a $2$-cocycle from $J$ to $V$. The symmetry
of $\delta f$ follows from 
\begin{equation*}
(\delta f)(x,y):=f(x\circ y)=f(y\circ x)=(\delta f)(y,x).
\end{equation*}%
Furthermore the Jordan identity of $J$ imply%
\begin{equation*}
(\delta f)(x^{2},x\circ y)=f(x^{2}\circ (x\circ y))=f(x\circ (x^{2}\circ
y))=(\delta f)(x,x^{2}\circ y).
\end{equation*}%
for all $x,y\in J.$ Hence a $2$-coboundary from $J$ to $V$ is a $2$-cocycle
from $J$ to $V$. Consequently, the map $\delta :f\longrightarrow \delta f$ \
is a map from $C^{1}(J,V)$ to $Z^{2}(J,V)$.\ 

\begin{definition}
The map%
\begin{equation*}
\delta :f\in C^{1}(J,V)\longrightarrow \delta f\in Z^{2}(J,V).
\end{equation*}%
where $\delta f$ is defined by $(\delta f)(x,y):=f(x\circ y)$ for all $%
x,y\in J,$is called the coboundary operator .
\end{definition}

Using the notions of $2$-cocycles and $1$-cochains, more precisely the
vector spaces $Z^{2}(J,V)$ and $C^{1}(J,V)$, one defines the second
cohomology group of a Jordan algebra $J$ by$V$ .

\begin{definition}
Let $J$ be a Jordan algebra,$V$ be avector space over $K.$ Then the quotient
vector space%
\begin{equation*}
H^{2}(J,V):=Z^{2}(J,V)\diagup \delta C^{1}(J,V)
\end{equation*}%
is called the second cohomology group of $J$ by $V$.
\end{definition}

Elements of $H^{2}(J,V)$ are equivalence classes of $2$-cocycles and two $2$%
-cocycles $\theta _{1}$ and $\theta _{2}$ are called equivalent cocycles if
they differ by a $2$-coboundary, i.e. if $\theta _{1}$ $=$ $\theta _{2}+$ $%
\delta f$ for some $f$ $\in $ $C^{1}(J,V)$. Equivalent $2$-cocycles are
called cohomologous.

\begin{lemma}
Let $\theta _{1}$ and $\theta _{2}$ be equivalent \ Jordan cocycles and \
let $J_{\theta _{1}},J_{\theta 2}$ be respectively the centeral extensions
constructed with these Jordan cocycles .Then the centeral extensions $%
J_{\theta _{1}}$ and $J_{\theta 2}$ are are equivalent extensions.
\end{lemma}

\begin{proof}
According to the definition of equivalent Jordan cocycles we have $\theta
_{1}=\theta _{2}+\delta f$ \ with $f\in C^{1}(J,V)$ a 1-cochain. Define $%
\sigma :J_{\theta _{2}}\longrightarrow J_{\theta _{2}+\delta f}$ by $\sigma
(x+v)=x+f(x)+v$ . Let $x+v\in \ker \sigma $ then $\sigma (x+v)=x+f(x)+v=0,$%
hence $x=v=0.$ Thus $\ker \sigma =\{0\},$ this shows that $\sigma $ is an
invertible linear transformation. Morever,for all $x+v,y+w\in J_{\theta
_{2}} $%
\begin{eqnarray*}
\sigma ((x+v)\circ _{J_{\theta _{2}}}(y+w)) &=&\sigma (x\circ _{J}y+\theta
_{2}(x,y)) \\
&=&x\circ _{J}y+f(x\circ _{J}y)+\theta _{2}(x,y) \\
&=&x\circ _{J}y+\delta f(x,y)+\theta _{2}(x,y) \\
&=&x\circ _{J}y+(\theta _{2}+\delta f)(x,y) \\
&=&x\circ _{J}y+\theta _{1}(x,y) \\
&=&(x+f(x)+v)\circ _{J_{\theta _{1}}}(y+f(y)+w) \\
&=&\sigma (x+v)\circ _{J_{\theta _{1}}}\sigma (y+w).
\end{eqnarray*}%
Then $J_{\theta _{1}}$ and $J_{\theta _{2}}$ are isomorphic. Hence,
cohomologous 2-cocycles yield equivalent (central) extensions.
\end{proof}

\begin{corollary}
A cohomology class $\theta \in H^{2}(J,V)$ defines a central extension of
the Jordan algebra $J$ by $V$ which is unique up to equivalence.
\end{corollary}

The cocycle $\theta =0$ $\in H^{2}(J,V)$ gives a trivial central extension
of a Jordan algebra $J$ by $V$. Namely%
\begin{equation}
J_{0}=J\oplus V  \tag*{ ( \ direct sum of Jordan algebras )}
\end{equation}%
and the Jordan multiplication is given by 
\begin{equation*}
(x+V)\circ _{J_{0}}(y+w)=x\circ _{J}y.
\end{equation*}%
for all $x,y\in J$ and $v,w\in V.$

For any $f\in C^{1}(J,V)$ a $1$-cochains, $\delta f=0+\delta f.$ Hence, $2$%
-cocycles which are obtained from $1$-cochains are cohomologous with the
trivial $2$-cocycle $\theta =0$ we obtain the following corollary.

\begin{corollary}
A central extension defined by a $2$-coboundary is equivalent with a trivial
central extension.
\end{corollary}

Let $V$ be $m$-dimensional vector space then $%
Z^{2}(J,V)=Z^{2}(J,K^{m})=Z^{2}(J,K)^{m}$ and $%
H^{2}(J,V)=H^{2}(J,K^{m})=H^{2}(J,K)^{m}.$ So for any $\theta \in Z^{2}(J,V)$
we may write $\theta =(\theta _{1},\theta _{2},.....,\theta _{m})\in
Z^{2}(J,K)^{m}.$ Let $e_{1},e_{2},..........,e_{m}$ be a basis of $V$ then $%
\theta (x,y)=\overset{m}{\underset{i=1}{\sum }}\theta _{i}(x,y)e_{i\text{ }}$
where $\theta _{i}\in Z^{2}(J,K)$ , and $\theta $ is a $2-$coboundary if and
only if all $\theta _{i}\in \delta C^{1}(J,K).$

\begin{lemma}
dim $\delta C^{1}(J,K)=$dim $J^{2}$
\end{lemma}

\begin{proof}
Let $J$ be n-dimensional Jordan algebra with basis $%
<x_{1},x_{2},...........,x_{n}>$ and $J^{\ast }$ generated by the dual basis 
$<x_{1}^{\ast },x_{2}^{\ast },...........,x_{n}^{\ast }>$ defined by $%
x_{i}^{\ast }(x_{i})=1$ and $x_{i}^{\ast }(x_{j})=0$ if $i\neq j.$Let $%
<x_{r},x_{r+1},...........,x_{r+s}>$ be a basis of $J^{2}.$Notice that $%
C^{1}(J,K)$ is the dual vector space of $J.$ The coboundary operator%
\begin{equation*}
\delta :J^{\ast }\longrightarrow Z^{2}(J,K)
\end{equation*}%
defined by 
\begin{equation*}
(\delta x_{k}^{\ast })(x_{i},x_{j})=x_{k}^{\ast }(x_{i}\circ x_{j})
\end{equation*}%
for $x_{k}^{\ast }\in J$ $^{\ast }$ and $x_{i},x_{j}\in J$. Let $x_{k}\notin
J^{2}$ then $\delta $ $x_{k}^{\ast }=0,$ so $\delta C^{1}(J,K)$ is spanned
by $<x_{r}^{\ast },x_{r+1}^{\ast },...........,x_{r+s}^{\ast }>$ which have
a dimension equal to the dimension of $J^{2}.$
\end{proof}

\section{Analouge of The Skjelbred-Sund Method}

Let $J$ be a Jordan algebra with non-trivial centre $Z(J)$. We will show
that such algebra are centeral extensions of smaller Jordan algebras .

\begin{lemma}
\label{h1}Every $n$-dimensional Jordan algebra with non-trivial centre is a
centeral extension of \ a lower dimensional Jordan algebra .
\end{lemma}

\begin{proof}
Let $J$ be $n-$dimensional Jordan algebra with non trivial centre$Z(J)$ ,
then $J/Z(J)$ is a Jordan algebra of smaller dimension . Let $\pi
:J\longrightarrow J/Z(J)$ be a projection map, choose\ an injective linear
map $\eta :J/Z(J)\longrightarrow J$ such that $\pi (\eta (\bar{x}))=\bar{x}$
for all $\bar{x}\in J/Z(J).$ For $\bar{x},\bar{y}\in J/Z(J)$ we have $\eta (%
\bar{x})\circ \eta (\bar{y})-\eta (\bar{x}\circ \bar{y})\in Z(J)$. This
define a symmetric bilinear map 
\begin{equation*}
\theta :J/Z(J)\times J/Z(J)\longrightarrow Z(J)
\end{equation*}
defined by $\theta (\bar{x},\bar{y})=\eta (\bar{x})\circ \eta (\bar{y})-\eta
(\bar{x}\circ \bar{y}).$ Then%
\begin{eqnarray*}
\theta (\bar{x}^{2},\bar{x}\circ \bar{y}) &=&\eta (\bar{x}^{2})\circ \eta (%
\bar{x}\circ \bar{y})-\eta (\bar{x}^{2}\circ (\bar{x}\circ \bar{y})) \\
&=&(\eta (\bar{x})\circ \eta (\bar{x})-\theta (\bar{x},\bar{x}))\circ (\eta (%
\bar{x})\circ \eta (\bar{y})-\theta (\bar{x},\bar{y}))-\eta (\bar{x}\circ (%
\bar{x}^{2}\circ \bar{y}) \\
&=&\eta (\bar{x})\circ (\eta (\bar{x})^{2}\circ \eta (\bar{y}))-\eta (\bar{x}%
\circ (\bar{x}^{2}\circ \bar{y}) \\
&=&\eta (\bar{x})\circ (\eta (\bar{x}^{2})\circ \eta (\bar{y}))-\eta (\bar{x}%
\circ (\bar{x}^{2}\circ \bar{y}) \\
&=&\theta (\bar{x},\bar{x}^{2}\circ \bar{y}).
\end{eqnarray*}%
Hence, $\theta $ is $2-$cocycle. It remains to show that $(J/Z(J))_{\theta
}=J/Z(J)\oplus Z(J)$ is isomorphic to $J$. Let $x\in J$, then x can uniquely
written as $x=\eta (\bar{y})+z$, where $\bar{y}\in J/Z(J)$ and $z\in Z(J).$
Define $\phi :J\longrightarrow (J/Z(J))_{\theta }$ by $\phi (x)=\bar{y}+z.$
Then $\phi $ is bijective and 
\begin{eqnarray*}
\phi (x_{1}\circ _{J}x_{2}) &=&\phi ((\eta (\bar{y}_{1})+z_{1})\circ
_{J}(\eta (\bar{y}_{2})+z_{2})) \\
&=&\phi (\eta (\bar{y}_{1})\circ _{J}\eta (\bar{y}_{2})) \\
&=&\phi (\eta (\bar{y}_{1}\circ _{J/Z(J)}\bar{y}_{2})+\theta (\bar{y}_{1},%
\bar{y}_{2})) \\
&=&\bar{y}_{1}\circ \bar{y}_{2}+\theta (\bar{y}_{1},\bar{y}_{2}) \\
&=&(\bar{y}_{1}+z_{1})\circ _{(J/Z(J))_{\theta }}(\bar{y}_{2}+z_{2}) \\
&=&\phi (x_{1})\circ _{_{(J/Z(J))_{\theta }}}\phi (x_{2})
\end{eqnarray*}%
Then $\phi $ is an isomorhism.
\end{proof}

So in particular, Every $n$-dimensional nilpotent Jordan algebras is a
centeral extension of a lower dimensional nilpotent Jordan algebra .

\begin{definition}
Let $\ \theta \in Z^{2}(J,V)$ then \ 
\begin{equation*}
\theta ^{\bot }=\{x\in J:\theta (x,y)=0\text{ for all }y\in J\}
\end{equation*}%
is called the radical of $\theta $ .
\end{definition}

Let us now fix a basis $\{e_{1},...,e_{r}\}$ of $V$ . A cocycle $\theta $ $%
\in Z_{2}(J,V)$ such that $\theta (x,y)=\overset{r}{\underset{i=1}{\sum }}%
\theta _{i}(x,y)e_{i\text{ }}$ where $\theta _{i}\in Z^{2}(J,K).$ Then $%
\theta ^{\perp }=\theta _{1}^{\perp }\cap \theta _{2}^{\perp }\cap ....\cap
\theta _{r}^{\perp }$

\begin{lemma}
Let $\ \theta \in Z^{2}(J,V)$ then $Z(J_{\theta })=(\theta ^{\bot }\cap
Z(J))\oplus V.$
\end{lemma}

\begin{proof}
Let $x+v\in Z(J_{\theta })$ then ($x+v)\circ _{J_{\theta }}(y+w)=0$ for all $%
y+w\in J_{\theta }.$Then $x\circ _{J}y+\theta (x,y)=0$ for all $y+w\in
J_{\theta }.$It follows that $x\in \theta ^{\bot }\cap Z(J)$ then $x+v\in
(\theta ^{\bot }\cap Z(J))\oplus V.$On the other hand , suppose that $x+v\in
(\theta ^{\bot }\cap Z(J))\oplus V$ then for all $y+w\in J_{\theta }$ we
have ($x+v)\circ _{J_{\theta }}(y+w)=x\circ _{J}y+\theta (x,y)=0$ then $%
x+v\in Z(J_{\theta }).$ Then $Z(J_{\theta })=(\theta ^{\bot }\cap
Z(J))\oplus V.$
\end{proof}

A subspace $W$ of $H^{2}(J,K)$ is said to be allowable if $\underset{%
_{\theta \in W}}{\cap }\theta ^{\perp }\cap Z(J)=0.$

\begin{corollary}
$\theta ^{\bot }\cap Z(J)=0$ if and only if $Z(J_{\theta })=V.$
\end{corollary}

If $\theta ^{\bot }\cap Z(J)\neq 0$ then $J_{\theta }$ can be obtaines as a
centeral extension of another Jordan algebra $\tilde{J}$ \ by $(\theta
^{\bot }\cap Z(J))\oplus V.$ Then to avoid constructing the same Jordan
algebra as a centeral extension of different Jordan algebra we want to
restrict to $\theta $ such that $\theta ^{\bot }\cap Z(J)=0$ .

Let $J$ be a Jordan over a field $K$ . For each $\theta \in Z^{2}(J,V)$ and $%
\phi \in Aut(J)$ ,the automorphism group of $J$ , we define $\phi \theta
(x,y)=\theta (\phi x,\phi y)$ for any $x,y\in J.$ So $Aut(J)$ acts on $%
Z^{2}(J,V)$ , and $\phi \theta \in \delta C^{1}(J,V)$ if and only if $\theta
\in \delta C^{1}(J,V).$ Then $Z^{2}(J,V)$ and $\delta C^{1}(J,V)$ are
invariant under the action of $Aut(J).$ So $Aut(J)$ acts on $H^{2}(J,V)$ .

Let $\theta _{1},\theta _{2}\in H^{2}(J,V)$ and $\theta _{1}^{\bot }\cap
Z(J)=\theta _{2}^{\bot }\cap Z(J)=0$ , i.e $Z(J_{\theta _{1}})=Z(J_{\theta
_{2}})=V.$ Assume that $J_{\theta _{1}}$ and $J_{\theta _{2}}$ are
isomorphic, let $\alpha :J_{\theta _{1}}\longrightarrow J_{\theta _{2}}$ be
an isomorphism .Dividing with the common center $V$ we obtain an
automorphism $\alpha _{0}:J\longrightarrow J.$ We can realize $\alpha $ as a
matrix to a suitable basis for $J\oplus V$ which assumed to contain a basis
for $J$ and a basis for $V$ :%
\begin{equation*}
\alpha =\left( 
\begin{array}{cc}
\begin{array}{c}
\alpha _{0} \\ 
\end{array}
& 
\begin{array}{c}
0 \\ 
\end{array}
\\ 
\varphi & \psi%
\end{array}%
\right)
\end{equation*}%
Where $\alpha _{0}\in Aut(J)$ , $\psi =\alpha \mid _{V}\in Gl(v)$ and $%
\varphi \in Hom(J,V).$

Now $\alpha $ preserves the Jordan products and writing $\circ _{1}$and $%
\circ _{2}$ for the Jordan products of $J_{\theta _{1}}$ and $J_{\theta
_{2}} $ respectively, we have%
\begin{equation*}
\alpha ((x+v)\circ _{1}(y+w))=\alpha (x+v)\circ _{2}\alpha (y+w)\text{ \ ; \ 
}x,y\in J\text{ and }v,w\in V.
\end{equation*}%
Where 
\begin{eqnarray*}
\alpha (x+v) &=&\alpha _{0}(x)+\varphi (x)+\psi (v) \\
\alpha (y+w) &=&\alpha _{0}(y)+\varphi (y)+\psi (w).
\end{eqnarray*}%
Then%
\begin{eqnarray*}
\alpha ((x+v)\circ _{1}(y+w)) &=&\alpha (x\circ y+\theta _{1}(x,y)) \\
&=&\alpha _{0}(x\circ y)+\varphi (x\circ y)+\psi (\theta _{1}(x,y)). \\
\alpha ((x+v)\circ _{2}\alpha (y+w)) &=&\alpha _{0}((x)+\varphi (x)+\psi
(v))\circ _{2}(\alpha _{0}(y)+\varphi (y)+\psi (w)) \\
&=&\alpha _{0}(x)\circ \alpha _{0}(y)+\theta _{2}(\alpha _{0}(x),\alpha
_{0}(y))
\end{eqnarray*}%
This yields%
\begin{eqnarray*}
\theta _{2}(\alpha _{0}(x),\alpha _{0}(y)) &=&\varphi (x\circ y)+\psi
(\theta _{1}(x,y))\ \ ,\ x,y\in J \\
\theta _{2}(\alpha _{0}(x),\alpha _{0}(y)) &=&\psi (\theta _{1}(x,y))\text{
\ mod }\delta C^{1}(J,V)
\end{eqnarray*}%
In case of $\theta _{1}=\theta _{2}=\theta $ we obtain the following
description of $Aut(J_{\theta })$

\begin{lemma}
\label{n}Let J be a nilpotent Jordan algebra .Let $\theta \in H^{2}(J,V)$
and $\theta ^{\bot }\cap Z(J)=0.$ Then the automorphism group $Aut(J_{\theta
})$ of the extension algebra $J_{\theta }$ consists of all linear operators
of the matrix form 
\begin{equation*}
\alpha =\left( 
\begin{array}{cc}
\begin{array}{c}
\alpha _{0} \\ 
\end{array}
& 
\begin{array}{c}
0 \\ 
\end{array}
\\ 
\varphi & \psi%
\end{array}%
\right)
\end{equation*}%
Where $\alpha _{0}\in Aut(J)$ , $\psi =\alpha \mid _{V}\in Gl(V)$ and $%
\varphi \in Hom(J,V).$ Such that :%
\begin{equation*}
\theta (\alpha _{0}(x),\alpha _{0}(y))=\varphi (x\circ y)+\psi \theta (x,y)\
,\forall \ x,y\in J.
\end{equation*}%
\begin{equation*}
\end{equation*}
\end{lemma}

Thus $J_{\theta _{1}}$ and $J_{\theta _{2}}$ are isomorphic if and only if
there exist $\alpha _{0}\in Aut(J)$ and $\psi \in Gl(V)$ such that $\alpha
_{0}\theta _{2}=\psi \theta _{1}$ mod $\delta C^{1}(J,V)$, i.e $\alpha
_{0}\theta _{2}$ and $\psi \theta _{1}$ are cohomologous. Let $\theta
_{1}=(\theta _{11},\theta _{12},.....,\theta _{1s})\in Z^{2}(J,V)$ , $\theta
_{2}=(\theta _{21},\theta _{22},.....,\theta _{2s})\in Z^{2}(J,V)$ and $%
e_{1},e_{2},..........,e_{m}$ be a basis of $V$ then $\theta _{1}(x,y)=%
\overset{s}{\underset{i=1}{\sum }}\theta _{1i}(x,y)e_{i\text{ }}$ and $%
\theta _{2}(x,y)=\overset{s}{\underset{i=1}{\sum }}\theta _{2i}(x,y)e_{i%
\text{ }}.$Suppose that $J_{\theta _{1}}$ and $J_{\theta _{2}}$ are
isomorphic , then $\overset{s}{\underset{i=1}{\sum }}\alpha _{0}\theta
_{2i}(x,y)e_{i\text{ }}=\overset{s}{\underset{i=1}{\sum }}\theta
_{1i}(x,y)\psi (e_{i\text{ }}).$ Let $\psi (e_{i\text{ }})=\overset{m}{%
\underset{j=1}{\sum }}a_{ji}e_{j}$ then $\overset{s}{\underset{i=1}{\sum }}%
\alpha _{0}\theta _{2i}(x,y)e_{i\text{ }}=\overset{s}{\underset{i=1}{\sum }}%
\overset{m}{\underset{j=1}{\sum }}\theta _{1i}(x,y)a_{ji}e_{j}$ \ hence : 
\begin{equation}
\alpha _{0}\theta _{2i}(x,y)=\overset{s}{\underset{i=1}{\sum }}\theta
_{1i}(x,y)a_{ji}\text{ \ \ mod }\delta C^{1}(J,V)  \label{k}
\end{equation}%
It follows that $\alpha _{0}\theta _{2i}$ span the same subspace of $%
H^{2}(J,K)$ as the $\theta _{1i}$ . We have proved

\begin{lemma}
Let $\theta =(\theta _{1},\theta _{2},.....,\theta _{m})$ and $\eta =(\eta
_{1},\eta _{2},.....,\eta _{m})\in H^{2}(J,V)$ and $\theta ^{\bot }\cap
Z(J)=\eta ^{\bot }\cap Z(J)=0$ . Then $J_{\theta }$ and $J_{\eta }$ are
isomorphic if and only if there exist $\phi \in Aut(J)$ such that $\phi
\theta _{i}$ span the same subspace of $H^{2}(J,K)$ as the $\eta _{i}$ .
\end{lemma}

i.e Two allowable subspaces give isomorphic Jordan algebras if and only if
they are in the same $Aut(J)$-orbit.%
\begin{equation*}
\end{equation*}

Let $J=I_{1}\oplus I_{2}$ be the direct sum of two ideals. Suppose that $%
I_{2}$ is contained in the centre of $J$. Then $I_{2}$ is called a centeral
component of $J$.%
\begin{equation*}
\end{equation*}

Jordan algebra $J=I_{1}\oplus I_{2}$ with centeral component $I_{2}$\ is a
trivial centeral extension of $I_{1}$ by $I_{2}.$ i.e Jordan algebras with
centeral components are simply obtained by taking direct sums of Jordan
algebras of smaller dimension with trivial Jordan algebras ( Vector spaces
). Therefore when constructing Jordan algebra with non trivial centre as
centeral extension we want to avoid constructing those with centeral
components. The folowing Lemma will help us to exclude Jordan algebras wit
centeral components .

\begin{lemma}
Let $\theta (x,y)=\underset{i=1}{\overset{r}{\sum }}$ $\theta
_{i}(x,y)e_{i}\in H^{2}(J,V)$ and $\theta ^{\bot }\cap Z(J)=0$ . Then $%
J_{\theta }$ has a centeral component if and only if \ $\theta _{1},\theta
_{2},.....,\theta _{r}$ are linearly dependent .
\end{lemma}

\begin{proof}
Suppose that $\theta _{1},\theta _{2},.....,\theta _{r}$ are linearly
dependent, then there exist a proper subset of $\theta _{1},\theta
_{2},.....,\theta _{r}$ span the same as $\theta _{i}.$ So we may assume
that that some of $\theta _{i}$ are zero , hence $J_{\theta }$ has a
centeral component .

On the other hand, Let $\theta (x,y)=\underset{i=1}{\overset{r=s+t}{\sum }}$ 
$\theta _{i}(x,y)e_{i}\in H^{2}(J,V)$ and $\theta ^{\bot }\cap Z(J)=0.$%
Suppose that $J_{\theta }=J\oplus V$ has a centeral component $B$ , and $B$
is contained in $V$ . Write $J_{\theta }=J\oplus W\oplus B$ , then $%
J_{\theta }$ can be considered as a trivial centeral extension of $\tilde{J}%
=J\oplus W$ by $B$ .The centre of $J_{\theta }$ is equal to $V=W\oplus B$
.Consider $\{e_{1},e_{2},.....,e_{s}\}$ and $\{e_{s+1},e_{s+2},.....,e_{t}\}$
are basis of $W\ $and $B$ respectively. Also $\tilde{J}=J\oplus W$ can be
considered as a centeral extension of $J$ by$W$ , then there exist a Jordan
cocycle $\psi (x,y)=\underset{i=1}{\overset{s}{\sum }}$ $\psi _{i}(x,y)e_{i}$
such that $\psi ^{\bot }\cap Z(J)=0\ $and $J_{\psi }=\tilde{J}=J\oplus W.$
Then $J_{\theta }$ is a centeral extension of $J_{\psi }$by $B$ with trivial
Jordan cocycle,, then 
\begin{equation*}
\theta (x,y)=\underset{i=1}{\overset{r=s+t}{\sum }}\theta
_{i}(x,y)e_{i}=\psi (x,y)=\underset{i=1}{\overset{s}{\sum }}\psi
_{i}(x,y)e_{i}
\end{equation*}%
Hence $\theta _{1},\theta _{2},.....,\theta _{r}$ are lineary dependent .
\end{proof}

\begin{equation*}
\end{equation*}

So to exclude centeral components, we must have $\theta _{i}$ to be linearly
independent .%
\begin{equation*}
\end{equation*}

A Jordan algebra $\tilde{J}$ is said to be a descendant of the Jordan
algebra $J$ if $\tilde{J}\diagup Z(\tilde{J})$ $\cong J$ and $Z(\tilde{J}%
)\leq $' $\tilde{J}^{2}$. If dim $Z(\tilde{J})$ $=r$ then $\tilde{J}$ is
also referred to as a step$-r$ descendant. A descendant of a nilpotent
Jordan algebra is nilpotent. Conversely, if $\tilde{J}$ is a
finite-dimensional nilpotent Jordan algebra over a field $K$, then by Lemma (%
\ref{h1}) $\tilde{J}$ is either a descendant of a smaller-dimensional
nilpotent Jordan algebra when $\tilde{J}$ has no centeral components, or $%
\tilde{J}$ $=J\oplus Kx$ where $J$ is an ideal of $\tilde{J}$ and $Kx$ is a
centeral component of $\tilde{J}$ ($Kx$ is a $1$-dimensional Vector space
viewed as trivial Jordan algebra).

All our previous observations can be summarized as follows:

\begin{theorem}
Let $J$ be a Lie algebra, let $V$be a vector space with fixed basis $%
\{e_{1},...,e_{r}\}$ over a field $K$, and let $\theta $, $\eta $ be
elements of $Z^{2}(J,V).$

\begin{enumerate}
\item The Jordan algebra $J_{\theta }$ is a step-$r$ descendant of $J$ if
and only if $\theta ^{\perp }$ $\cap Z(J)$ $=0$ and the image of the
subspace $<\theta _{1},...,\theta _{r}>$ in $H^{2}(J,K)$ is $r$-dimensional.

\item Suppose that $\eta $ is an other element of $Z^{2}(J,V)$ and that $%
J_{\theta },J_{\eta }$ are descendants of $J$. Then $J_{\theta }\cong
J_{\eta }$ if and only if images of the subspaces $<\theta _{1},...,\theta
_{r}>$and $<\eta _{1},...,\eta _{r}>$ in $H^{2}(J,K)$ are in the same orbit
under the action of $Aut(J)$.%
\begin{equation*}
\end{equation*}
\end{enumerate}
\end{theorem}

It follows that there is a one-to-one correspondence between the set of
isomorphism types of step-$r$ descendants of $J$ and the Aut($J$)-orbits on
the $r$-dimensional allowable subspaces of $H^{2}(J,K)$. Hence the
classification of $n$-dimensional nilpotent Jordan algebras requires that we
determine these orbits for all nilpotent Jordan algebras of dimension at
most $n-1$.

Let $G_{r}(H^{2}(J,K))$ be the Grassmanian of subspaces of dimension $r$ in $%
H^{2}(J,K)$ . There is anatural action of $Aut(J)$ on $G_{r}(H^{2}(J,K))$
defined by : 
\begin{equation*}
\phi <\theta _{1},\theta _{2},.....,\theta _{r}>=<\phi \theta _{1},\phi
\theta _{2},.....,\phi \theta _{r}>
\end{equation*}%
for $V=<$ $\theta _{1},\theta _{2},.....,\theta _{r}>\in G_{r}(H^{2}(J,K))$
and $\phi \in Aut(J)$ .

Note that if $\ \{\theta _{1},\theta _{2},.....,\theta _{r}\}$ is linear
independent so is$\{\phi \theta _{1},\phi \theta _{2},.....,\phi \theta
_{r}\}$.Define 
\begin{equation*}
U_{r}(J)=\{V=<\theta _{1},\theta _{2},.....,\theta _{r}>\in
G_{r}(H^{2}(J,K)):\theta _{i}^{\bot }\cap Z(J)=0;i=1,2,...,r\}.
\end{equation*}

\begin{lemma}
$U_{r}(J)$ is stable under the action of $Aut(J).$
\end{lemma}

\begin{proof}
Let $\phi \in Aut(J)$ and $V=<\theta _{1},\theta _{2},.....,\theta _{r}>\in
U_{r}(J).$ Let $x\in (\phi \theta _{i})^{\perp }=\{y\in J:\phi \theta
_{i}(y,J)=0\}$ then $\phi (x)\in \theta _{i}^{\perp }$ hence $x\in \phi
^{-1}\theta _{i}^{\perp }=\{\phi ^{-1}(y):y\in \theta _{i}^{\perp }\}$ .
Also let $\phi ^{-1}(x)\in \phi ^{-1}\theta _{i}^{\perp }$ then $\phi \theta
_{i}(\phi ^{-1}(x),J)=\theta _{i}(x,J)=0$ hence $\phi ^{-1}(x)\in (\phi
\theta _{i})^{\perp }.$ Then $(\phi \theta _{i})^{\perp }=\phi ^{-1}\theta
_{i}^{\perp }$ and $\phi ^{-1}(Z(J))=Z(J).$ Therefore $(\phi \theta
_{i})^{\bot }\cap Z(J)=\phi ^{-1}(\theta _{i}^{\bot }\cap Z(J))=0$ then $%
\phi V\in U_{r}(J).$ \ \ 
\end{proof}

\begin{equation*}
\end{equation*}

Let $U_{r}(J)\diagup Aut(J)$ be the set of $Aut(J)$-orbits of $U_{r}(J).$
Then there exists a canonical one-to-one correspondense from $%
U_{r}(J)\diagup Aut(J)$ \ onto the set of isomorphism classes of Jordan
algebras without centeral components which are central extensions of $J$ by $%
V$ and have $r$-dimensional center where $r=\dim V$.%
\begin{equation*}
\end{equation*}

All our previous observations can be summarized to have an analogue of the 
\textit{Skejelbred-Sund theorem} for Jordan algebras as follows:

\begin{theorem}
\label{h2}Let $J$ be a Jordan algebra over a field $K$.The isomorphism
clases of Jordan algebras $\tilde{J}$ with centre $V$ of dimension $r$ ,$%
\tilde{J}\diagup V\cong J$\ and without centeral component are in bijective
correspondence with the elements in $U_{r}(J)\diagup Aut(J).$%
\begin{equation*}
\end{equation*}
\end{theorem}

By this theorem, we may construct all nilpotent Jordan algebras of dimension 
$n$, given those algebras of dimension less than $n$, by centeral extension.

\section{Constructing nilpotent Jordan algebras}

In this section we carry out a procudure for constructing nilpotent Jordan
algebras over any arbitrary field.%
\begin{equation*}
\end{equation*}

\underline{\underline{\textbf{The classification procedure :}}}

\bigskip

Let $J$ be a Jordan algebras with basis $e_{1},e_{2},....,e_{n-r}$ then the
dual basis of the space of all symmetric bilinear forms $\theta :J\times
J\longrightarrow K$ are $\sum_{e_{i},e_{j}}$ $:n-r\geq i\geq j=1$\ with $%
\sum_{e_{i},e_{j}}(e_{l},e_{m})=\sum_{e_{i},e_{j}}(e_{m},e_{l})=1$ if $i=l$
and $j=m$ ,and takes the value $0$ otherwise, then any $2$-cocycle $\theta
\in $\ $Z^{2}(J,K)$ can be represented by $\theta =\underset{i\geq j=1}{%
\overset{n-r}{\sum }}c_{ij}\sum_{e_{i},e_{j}}$ such that the elements $%
c_{i,j}$ satisfy Jordan identity for $2-$cocycles . By Theorem (\ref{h2})\
we have a procdure that takes as input a Jordan algebra $J$ of dimension $%
n-r $ it outputs all nilpotent Jordan algebras $\tilde{J}$ of dimension $n$
such that $\tilde{J}\diagup Z($ $\tilde{J})\cong J$, and $\tilde{J}$ has no
centeral components. It runs as follows :

\begin{itemize}
\item For a given nilpotent Jordan algebra $J$ of dimension $n-r$, we list
at first its center to help us identify the $2$-cocycles satisfying $\theta
^{\perp }$ $\cap Z(J)$ $=0.$

\item Compute $Z^{2}(J,K)$: When computing the $2$-cocycles just list all
the constraints on the elements $c_{ij}.$

\item Compute $\delta C^{1}(J,K):$ For $e_{l}\in J^{2}$ define $\theta _{l}=%
\underset{i,j}{\sum }\left( \sum_{e_{i},e_{j}}\right) $ such that $%
e_{i}\circ e_{j}=e_{l}.$ Then $\delta C^{1}(J,K)$ is spanned by $\theta _{l}$
for all $e_{l}\in J^{2},$

\item Compute $H^{2}(J,K):$ The complement of $\delta C^{1}(J,K)$ in $%
Z^{2}(J,K)$ be $H^{2}(J,K).$ For $\theta =\underset{i\geq j=1}{\overset{n-r}{%
\sum }}c_{i,j}\sum_{e_{i},e_{j}}\in $ $Z^{2}(J,K)$ put $c_{i,j}=0$ for all $%
\sum_{e_{i},e_{j}}=\theta _{l}\in $ $\delta C^{1}(J,K).$

\item Consider $\theta \in H^{2}(J,V)$ with $\theta (x,y)=\overset{r}{%
\underset{i=1}{\sum }}\theta _{i}(x,y)e_{n-r+i}$ where the $\theta _{i}\in
H^{2}(J,K)$ are linearly indpendent, and $\theta ^{\perp }$ $\cap Z(J)$ $=0.$
Find a list of representatives of the orbits of $Aut(J)$ acting on the $%
\theta $.

\item For each $\theta $ found, construct $J_{\theta }$. Discard the
isomorphic ones.%
\begin{equation*}
\end{equation*}
\end{itemize}

\textbf{Remarks :}

\begin{enumerate}
\item If charcteristic $K\neq 2,$ one can linearize $\theta (x^{2},x\circ
y)=\theta (x,x^{2}\circ y)$ to obtain :%
\begin{equation*}
\theta (x,v\circ (y\circ z))+\theta (y,v\circ (x\circ z))+\theta (z,v\circ
(x\circ y))=\theta (x\circ y,z\circ v)+\theta (y\circ z,x\circ v)+\theta
(x\circ z,y\circ v)
\end{equation*}%
for all $x,y,z,v\in J.$ 

\item One can use Lemma $\left( \ref{n}\right) $ to compute the automorphism
group,and the system of equations $\left( \ref{k}\right) $for testing if two
cocycles are in the same $Aut(J)$-orbit by checking solvalbality. However in
the next section we construct an algorithm for testing isomorphism of any
two Jordan algebra, by this algorithm we can also compute the automorphism
group.

\item The procedure only gives those nilpotent Jordan algebras without
cetral components. So we have to add the nilpotent Jordan algebras obtained
by by taking direct sum of a smaller dimensional nilpotent Jordan algebra
with trivial Jordan algebra (that has trivial multiplication).
\end{enumerate}

\section{Deciding isomorphism of Jordan algebras}

A classical problem is to know how many different (up to isomorphisms)
finite dimensional Jordan algebras exist for each dimension. Let $J$ be a
Jordan algebra of dimension $n$, then it has a basis $%
e_{1},e_{2},....,e_{n}. $ It follows that there are constants $c_{ij}^{k}$
for $n\geq i\geq j=1$ such that%
\begin{equation*}
e_{i}\circ e_{j}=\underset{k=1}{\overset{n}{\sum }}c_{ij}^{k}e_{k}.
\end{equation*}%
Let $J_{1}$ and $J_{2}$ be two Jordan algebras for which we want to decide
whether they are isomorphic or not. If the dimension of $J_{1}$ and $J_{2}$
are not equal, then $J_{1}$ and $J_{2}$ can never be isomorphic.
Furthermore, suppose that $\phi :J_{1}\longrightarrow J_{2}$ is an
isomorphism of Jordan algebras, then $\phi (Z(J_{1}))=Z(J_{2})$ and $\phi
(C^{k}(J_{1}))=C^{k}(J_{2}).$ So if $\dim Z(J_{1})\neq \dim Z(J_{2})$ then $%
J_{1}$ and $J_{2}$ can never be isomorphic. Also if $\dim C^{k}(J_{1})\neq
\dim C^{k}(J_{2})$ \ hen $J_{1}$ and $J_{2}$ can never be isomorphic. By
this way we may able to decide that $J_{1}$ and $J_{2}$ are not isomorphic.
There is a direct method for testing isomorphism of Jordan algebras by using
Grobner bases. Let $e_{1},e_{2},....,e_{n}$ be a basis of $J_{1}$ and let $%
\tilde{e}_{1},\tilde{e}_{2},....,\tilde{e}_{n}$ be a basis of $J_{2}.$ Let $%
(c_{ij}^{k})$ and $(\gamma _{ij}^{k})$ be the structure constants of $J_{1}$
and $J_{2}$, respectively. A map $\phi :J_{1}\longrightarrow J_{2}$ is an
isomorphism of Jordan algebras if and only if it satisfies the following
requirements :

\begin{itemize}
\item $\phi (e_{i}\circ e_{j})=\phi (e_{i})\circ \phi (e_{j})$ \ for $n\geq
i\geq j=1,$

\item $\phi $ is non-singular.
\end{itemize}

Let $\phi :J_{1}\longrightarrow J_{2}$ is an isomorphism of Jordan algebras
given by $\phi (e_{i})=\overset{n}{\underset{j=1}{\sum }}a_{ij}\tilde{e}_{j}$
then 
\begin{equation*}
\phi (e_{i}\circ e_{j})=\phi \left( \underset{k=1}{\overset{n}{\sum }}%
c_{ij}^{k}e_{k}\right) =\underset{k=1}{\overset{n}{\sum }}c_{ij}^{k}\phi
(e_{k})=\underset{k,m=1}{\overset{n}{\sum }}c_{ij}^{k}a_{km}\tilde{e}_{m},
\end{equation*}%
and%
\begin{eqnarray*}
\phi (e_{i})\circ \phi (e_{j}) &=&\left( \overset{n}{\underset{k=1}{\sum }}%
a_{ik}\tilde{e}_{k}\right) \circ \left( \overset{n}{\underset{l=1}{\sum }}%
a_{jl}\tilde{e}_{l}\right) \\
&=&\overset{n}{\underset{k,l=1}{\sum }}a_{ik}a_{jl}(\tilde{e}_{k}\circ 
\tilde{e}_{l})=\overset{n}{\underset{k,l,m=1}{\sum }}\gamma
_{kl}^{m}a_{ik}a_{jl}\tilde{e}_{m}.
\end{eqnarray*}%
Hence the first requirement amounts to the following $\frac{n(n+1)}{2}$
equations in the variables $a_{ij}$ :%
\begin{equation*}
\underset{k=1}{\overset{n}{\sum }}c_{ij}^{k}a_{km}-\overset{n}{\underset{%
k,l=1}{\sum }}\gamma _{kl}^{m}a_{ik}a_{jl}=0\text{ for }n\geq i\geq j=1\text{
and }1\leq m\leq n.
\end{equation*}%
By the second erquirement the determinant of the matrix $(a_{ij})$ not equal
to $0$, $\det (a_{ij})\neq 0.$

\begin{theorem}
(\textbf{Generalized Weak Nullstellensatz}). Let $k$ be any field. A system
of polynomials $f_{1},...,f_{m}$ $\in $ $k[x1,....,xn]$has no common zero
over the algebraic closure of $k$, if and only if, $1\in I$ where $I$ be the
ideal generated by the polynomials $f_{1},...,f_{m}.$
\end{theorem}

\bigskip

So we consider the ideal $I$ generated by the polynomials 
\begin{equation*}
\underset{k=1}{\overset{n}{\sum }}c_{ij}^{k}a_{km}-\overset{n}{\underset{%
k,l=1}{\sum }}\gamma _{kl}^{m}a_{ik}a_{jl}=0\text{ for }n\geq i\geq j=1\text{
and }1\leq m\leq n.
\end{equation*}%
together with the polynomial $b\det (a_{ij})-1$ in the polynomial ring $%
k[a_{ij},b].$ Then there is a solution of this system if and only if $I\neq
k[a_{ij},b].$ The algorithm for calculating Grobener bases yields a method
for deciding whether or not $1\in I.$ Then by Grobener bases calculations we
can decide isomorphism of Jordan algebras.

\begin{algorithm}
\textbf{Testing isomorphism for Jordan algebras}.

Input : Two Jordan algebras $(J_{1},\circ _{1})$ with basis $%
e_{1},e_{2},....,e_{n}$\ and $(J_{2},\circ _{2})$\ with basis $\tilde{e}_{1},%
\tilde{e}_{2},....,\tilde{e}_{n}.$

Output ; True if $J_{1}\cong J_{2}$ and false in other case.

\begin{enumerate}
\item Compute the following system of equations : 
\begin{equation*}
\phi (e_{i}\circ _{1}e_{j})-\phi (e_{i})\circ _{2}\phi (e_{j})=0\ \text{for }%
n\geq i\geq j=1.
\end{equation*}%
where $\phi =(a_{ij})_{n\times n}$

\item To ensure that $\phi $ is going to be non-singular, add the following
relation with a new variable $b$ :%
\begin{equation*}
b\det \phi -1=0.
\end{equation*}

\item Compute a grobener bases $G$ of the ideal $I=<\{\phi (e_{i}\circ
_{1}e_{j})-\phi (e_{i})\circ _{2}\phi (e_{j})\}_{n\geq i\geq j=1}\cup
\{b\det \phi -1\}>$ in the polynomial ring $k[a_{ij},b].$

\item $G=\{1\}?$

\begin{enumerate}
\item Yes.

Return False.

\item No.

Output : True ; Return $G$.
\end{enumerate}
\end{enumerate}
\end{algorithm}

If the output of the algorithm is true, then it provides the equations of an
algebraic variety whose points are all possiple values for $\phi ,$ hence we
have a descripiton of the automorphism group in the case $J_{1}=J_{2}$ .

\section{Nilpotent Jordan algebras of dimension$\leq 4$}

In this section we use the analouge \textbf{Skjelbred-Sund method} to
classify nilpotent Jordan algebras of dimension $\leq 3$ over any field, and
four dimensional nilpotent Jordan algebras in the following cases :

\begin{itemize}
\item Over algebraic closed field $K$ and $ch(K)\neq 2.$

\item Over the real field $%
%TCIMACRO{\U{211d} }%
%BeginExpansion
\mathbb{R}
%EndExpansion
.$
\end{itemize}

Here we denote the $j$-th algebra of dimension $i$ by $J_{i,j}$ .

\subsection{Nilpotent Jordan algebras of dimension $1$}

Let $J$ be a Jordan algebra of dimension $1$ spanned by a. If $J$ is
nilpotent then $Z(J)$ is non-trivial. Hence $Z(J)$ is spanned by $a$. It
follows that there is only one nilpotent Jordan algebra of dimension $1$
over any arbitrary field,$J_{1,1},$ it spanned by $a$ and $a^{2}=0$.

\begin{equation*}
\begin{tabular}{|l|l|}
\hline
$J_{1,1}$ & All multiplications are zero. \\ \hline
\end{tabular}%
\end{equation*}

\subsection{Nilpotent Jordan algebras of dimension $2$}

To construct all nilpotent Jordan algebras of dimension $2,$ firstly we
consider the trivial $1$-dimensional Central extension of the nilpotent
Jordan algebras of dimension $1$ corresponding to $\theta =0$ ( Algebras
that are direct sum of an algebra of dimension $1$ and a $1$-dimensional
trivial Jordan algebra, isomorphic to $J_{1,1}$) to get nilpotent Jordan
algebras with \textit{centeral component} . Finally consider non trivial $1$%
- dimensional Central extension of the nilpotent Jordan algebras of
dimension $1$ to get nilpotent Jordan algebras without \textit{centeral
component}.

\subsubsection{Nilpotent Jordan algebras with Centeral component}

\begin{enumerate}
\item \textbf{1- dimensional Trivial Central extension of }$J_{1,1}$\textbf{%
\ corresponding to }$\theta =0.$

\begin{enumerate}
\item We get the Jordan algebra :%
\begin{equation*}
\begin{tabular}{|l|l|}
\hline
$J_{2,1}=J_{1,1}\oplus J_{1,1}$ & All multiplications are zero. \\ \hline
\end{tabular}%
\end{equation*}
\end{enumerate}
\end{enumerate}

\subsubsection{Nilpotent Jordan algebras without Centeral component}

\begin{enumerate}
\item 1\textbf{- dimensional Central extension of }$J_{1,1}$

Here we get that $Z^{2}(J_{1,1},K)$ is spanned by $\sum_{a,a}$ and $\delta
C^{1}(J_{1,1},K)=0$, then $H^{2}(J_{1,1},K)$ is spanned by $\sum_{a,a}.$ So
we get only one cocycle $\theta =\sum_{a,a}$, yielding the algebra :%
\begin{equation*}
\begin{tabular}{|l|l|}
\hline
$J_{2,2}$ & $a^{2}=b.$ \\ \hline
\end{tabular}%
\end{equation*}
\end{enumerate}

\begin{theorem}
Up to isomorphism there exactly two nilpotent Jordan algebras of dimension $%
2 $ over any field $K$ which are isomorphic to the following nilpotent
Jordan algebras :%
\begin{equation*}
\begin{tabular}{|l|}
\hline
\textbf{Nilpotent Jordan algebras with Centeral component} \\ \hline
$J_{2,1}=J_{1,1}\oplus J_{1,1}$. \\ \hline
\textbf{Nilpotent Jordan algebras without Centeral component} \\ \hline
$J_{2,2}$ $:$ $a^{2}=b.$ \\ \hline
\end{tabular}%
\end{equation*}
\end{theorem}

\subsection{Nilpotent Jordan algebras of dimension $3$}

From the previous subsection we get that the $2$-dimensional nilpotent
Jordan algebras over any field $K$ are : $J_{2,1}$ and $J_{2,2}.$ Using
those algebra to construct $3$-dimensional nilpotent Jordan algebras as
follow :

\subsubsection{Nilpotent Jordan algebras with Centeral component}

\begin{enumerate}
\item \textbf{\ 1- dimensional Trivial Central extension of }$J_{2,1}$%
\textbf{\ corresponding to }$\theta =0.$

We get the Jordan algebra :%
\begin{equation*}
\begin{tabular}{|l|l|}
\hline
$J_{3,1}=J_{2,1}\oplus J_{1,1}$ & All multiplications are zero. \\ \hline
\end{tabular}%
\end{equation*}

\item 1\textbf{- dimensional Trivial Central extension of }$J_{2,2}$\textbf{%
\ corresponding to }$\theta =0.$
\end{enumerate}

We get the Jordan algebra :%
\begin{equation*}
\begin{tabular}{|l|l|}
\hline
$J_{3,2}=J_{2,2}\oplus J_{1,1}$ & $a^{2}=b$. \\ \hline
\end{tabular}%
\end{equation*}

\subsubsection{Nilpotent Jordan algebras without Centeral component}

\begin{enumerate}
\item 1\textbf{- dimensional Central extension of }$J_{2,1}$\textbf{\ }

Here we get that $Z^{2}(J_{2,1},K)$ is spanned by $\sum_{a,a}$,$\sum_{a,b}$
and $\sum_{b,b},$ also we get that $\delta C^{1}(J_{2,1},K)=0$. Then $%
H^{2}(J_{2,1},K)$ is consists of $\theta =\alpha _{1}\sum_{a,a}+\alpha
_{2}\sum_{a,b}+\alpha _{3}\sum_{b,b}.$ The centre of $J_{2,1},Z(J_{2,1}),$
spanned by $a$ and $b$, it follows that $\theta :J_{2,1}\times
J_{2,1}\longrightarrow K$ such that%
\begin{equation*}
\theta ^{\perp }\cap Z(J_{2,1})=0
\end{equation*}%
is nondegenerate since $\theta ^{\perp }=0.$ The autommorphism group
consists of $:$%
\begin{equation*}
\phi =\left( 
\begin{array}{cc}
a_{11} & a_{12} \\ 
a_{21} & a_{22}%
\end{array}%
\right) ,\det \phi \neq 0.
\end{equation*}

\underline{Characteristic $K\neq 2$}

Every symmetric bilinear form is diagonalizable, so we may assume that $%
\theta =\sum_{a,a}+\alpha \sum_{b,b}$ and $\alpha \neq 0.$ Write $\phi
\theta =\alpha _{1}^{\prime }\sum_{a,a}+\alpha _{2}^{\prime
}\sum_{a,b}+\beta \sum_{b,b},$ then 
\begin{eqnarray*}
\alpha _{1}^{\prime } &=&a_{11}^{2}+\alpha a_{21}^{2} \\
\alpha _{2}^{\prime } &=&a_{11}a_{12}+\alpha a_{21}a_{22} \\
\beta &=&a_{12}^{2}+\alpha a_{22}^{2}
\end{eqnarray*}%
To fix $\alpha _{1}^{\prime }=1$ and $\alpha _{2}^{\prime }=0,$ choose $%
a_{21}=a_{12}=0$ and $a_{11}=1.$ Then $\phi \left( \sum_{a,a}+\alpha
\sum_{b,b}\right) =\sum_{a,a}+\alpha a_{22}^{2}\sum_{b,b}.$ This show that
for any $\alpha ,\beta ,a_{22}\in K^{\ast }$ such that $\beta =\alpha
a_{22}^{2},$ then $\theta _{1,\alpha }$ and $\theta _{1,\beta }$ are in the
same $Aut(J_{2,1})$-orbit. On the other hand , suppose that $\theta
_{1,\alpha }$ and $\theta _{1,\beta }$ are in the same $Aut(J_{2,1})$-orbit.
Then there exist $\phi \in Aut(J_{2,1})$ and $\lambda \in K^{\ast }$ such
that $\phi \theta _{1,\alpha }=\lambda \theta _{1,\beta },$ so we have 
\begin{equation*}
\phi ^{t}\left( 
\begin{array}{cc}
1 & 0 \\ 
0 & \alpha%
\end{array}%
\right) \phi =\allowbreak \lambda \left( 
\begin{array}{cc}
1 & 0 \\ 
0 & \beta%
\end{array}%
\right) \allowbreak
\end{equation*}%
it follows that $\beta =\left( \frac{\det \phi }{\lambda }\right) ^{2}\alpha
.$ Hence $\theta _{1,\alpha }$ and $\theta _{1,\beta }$ are in the same $%
Aut(J_{2,1})$-orbit if and only if there is a $\delta \in K^{\ast }$ with $%
\beta =\delta ^{2}\alpha .$ So we get the algebra%
\begin{equation*}
\begin{tabular}{|l|l|}
\hline
$J_{3,3}^{\alpha }$ & $a^{2}=c$ , $b^{2}=\alpha c$ , $\alpha \in K^{\ast
}\diagup \left( K^{\ast }\right) ^{2}.$ \\ \hline
\end{tabular}%
\end{equation*}

Over algebraic closed field $K$, $K^{\ast }\diagup \left( K^{\ast }\right)
^{2}=\{1\}.$ So we get only one 
\begin{equation*}
\begin{tabular}{|l|l|}
\hline
$J_{3,3}$ & $a^{2}=c$ , $b^{2}=c.$ \\ \hline
\end{tabular}%
\end{equation*}

Over real field $%
%TCIMACRO{\U{211d} }%
%BeginExpansion
\mathbb{R}
%EndExpansion
,$ $%
%TCIMACRO{\U{211d} }%
%BeginExpansion
\mathbb{R}
%EndExpansion
^{\ast }\diagup \left( 
%TCIMACRO{\U{211d} }%
%BeginExpansion
\mathbb{R}
%EndExpansion
^{\ast }\right) ^{2}=\{\pm 1\}.$ So we get two algebras%
\begin{equation*}
\begin{tabular}{|l|l|}
\hline
$J_{3,3}$ & $a^{2}=c$ , $b^{2}=c.$ \\ \hline
$J_{3,3}^{-1}$ & $a^{2}=c$ , $b^{2}=-c.$ \\ \hline
\end{tabular}%
\end{equation*}

Over finite field $K,K^{\ast }\diagup \left( K^{\ast }\right) ^{2}=\{\pm
1\}. $ So we get two algebras%
\begin{equation*}
\begin{tabular}{|l|l|}
\hline
$J_{3,3}$ & $a^{2}=c$ , $b^{2}=c.$ \\ \hline
$J_{3,3}^{-1}$ & $a^{2}=c$ , $b^{2}=-c.$ \\ \hline
\end{tabular}%
\end{equation*}

Over rational field $%
%TCIMACRO{\U{211a} }%
%BeginExpansion
\mathbb{Q}
%EndExpansion
,%
%TCIMACRO{\U{211a} }%
%BeginExpansion
\mathbb{Q}
%EndExpansion
^{\ast }\diagup \left( 
%TCIMACRO{\U{211a} }%
%BeginExpansion
\mathbb{Q}
%EndExpansion
^{\ast }\right) ^{2}$ is infinite. Then there is an infinite number of them.

\underline{Characteristic $K=2$}

If $\theta $ is\ alternate we obtain only one cocycle $\theta
_{1}=\sum_{a,b} $. On the other hand if $\theta $ is not alternate, then it
is diagonalizable. In this case $K^{\ast }\diagup \left( K^{\ast }\right)
^{2}=\{1\}$. so we obtain only one cocycle $\theta
_{2}=\sum_{a,a}+\sum_{b,b}.$ Clearly $\theta _{1},\theta _{2}$ are
inequivalent bilinear form, hence they not lie in the same $Aut(J_{2,1})$%
-orbit,[ $\phi \theta _{1}=(\det \phi )\theta _{1}.$ i.e $\phi \theta _{1}$
is a multiple of $\theta _{1},$hence it is not conjugate to $\theta _{2}.$
Then$\left( J_{2,1}\right) _{\theta _{1}}$ and $\left( J_{2,1}\right)
_{\theta _{2}}$ are not isomorphic]. We get two algebras 
\begin{equation*}
\begin{tabular}{|l|l|}
\hline
$J_{3,3}$ & $a^{2}=c$ , $b^{2}=c.$ \\ \hline
$J_{3,4}$ & $a\circ b=c.$ \\ \hline
\end{tabular}%
\end{equation*}

\item 1\textbf{- dimensional Central extension of }$J_{2,2}$

Here we get that $Z^{2}(J_{2,1},K)$ is spanned by $\sum_{a,a}$ and $%
\sum_{a,b}.$ Moreover $\delta C^{1}(J_{1,1},K)$ is spanned by $\sum_{a,a}$,
then $H^{2}(J_{1,1},K)$ is spanned by $\sum_{a,b}.$ So we get only one
cocycle $\theta =\sum_{a,b}$, yielding the algebra%
\begin{equation*}
\begin{tabular}{|l|l|}
\hline
$J_{3,4}$ & $a^{2}=b$ , $a\circ b=c$. \\ \hline
\end{tabular}%
\end{equation*}

\item 2\textbf{- dimensional Central extension of }$J_{1,1}$

We have that $H^{2}(J_{1,1},K)$ is 1-dimensional spanned by $\sum_{a,a},$
then there is no 2- dimensional Central extension of $J_{1,1}.$
\end{enumerate}

We can summarize some of our results in the following theorems :

\begin{theorem}
Up to isomorphism there exist $4$ nilpotent Jordan algebras of dimension $3$
over algebraic closed field $K$\ and $ch(K)\neq 2$ which are isomorphic to
one of the following \ pairwise non-isomorphic nilpotent Jordan algebras :%
\begin{equation*}
\begin{tabular}{|l|}
\hline
\textbf{Nilpotent Jordan algebras with Centeral component} \\ \hline
$J_{3,1}=J_{2,1}\oplus J_{1,1}.$ \\ \hline
$J_{3,2}=J_{2,2}\oplus J_{1,1}.$ \\ \hline
\textbf{Nilpotent Jordan algebras without Centeral component} \\ \hline
$J_{3,3}$ $:$ $a^{2}=c$ , $b^{2}=c.$ \\ \hline
$J_{3,4}$ $:$ $a^{2}=b$ , $a\circ b=c$. \\ \hline
\end{tabular}%
\end{equation*}
\end{theorem}

\begin{theorem}
Up to isomorphism there exist $5$ nilpotent Jordan algebras of dimension $3$
over field $K$\ and $ch(K)=2$ which are isomorphic to one of the following \
pairwise non-isomorphic nilpotent Jordan algebras :%
\begin{equation*}
\begin{tabular}{|l|}
\hline
\textbf{Nilpotent Jordan algebras with Centeral component} \\ \hline
$J_{3,1}=J_{2,1}\oplus J_{1,1}.$ \\ \hline
$J_{3,2}=J_{2,2}\oplus J_{1,1}.$ \\ \hline
\textbf{Nilpotent Jordan algebras without Centeral component} \\ \hline
$J_{3,3}$ $:$ $a\circ b=c.$ \\ \hline
$J_{3,4}$ $:$ $a^{2}=c$ , $b^{2}=c.$ \\ \hline
$J_{3,5}$ $:$ $a^{2}=b$ , $a\circ b=c$. \\ \hline
\end{tabular}%
\end{equation*}
\end{theorem}

\begin{theorem}
Up to isomorphism there exist $5$ nilpotent Jordan algebras of dimension $3$
over $%
%TCIMACRO{\U{211d} }%
%BeginExpansion
\mathbb{R}
%EndExpansion
$ which are isomorphic to one of the following \ pairwise non-isomorphic
nilpotent Jordan algebras :%
\begin{equation*}
\begin{tabular}{|l|}
\hline
\textbf{Nilpotent Jordan algebras with Centeral component} \\ \hline
$J_{3,1}=J_{2,1}\oplus J_{1,1}.$ \\ \hline
$J_{3,2}=J_{2,2}\oplus J_{1,1}.$ \\ \hline
\textbf{Nilpotent Jordan algebras without Centeral component} \\ \hline
$J_{3,3}^{\alpha =\pm 1}$ $:$ $a^{2}=c$ $,$ $b^{2}=\alpha c.$ \ \ \ \ \ \ $%
\left( \alpha =\pm 1\right) $ \\ \hline
$J_{3,4}$ $:$ $a^{2}=b$ $,$ $a\circ b=c.$ \\ \hline
\end{tabular}%
\end{equation*}
\end{theorem}

We are showed that also the number of nilpotent Jordan algebras of dimension 
$3$ over finite field is five for any characteristic.

\begin{theorem}
Every nilpotent Jordan algebra of dimension $\leq 3$ is a commutative
nilpotent associative algebra.
\end{theorem}

It follows that all commutative nilpotent associative algebras of dimension $%
\leq 3$ are classified, and all previous theorems are valid for commutative
nilpotent associative algebras.

\subsection{Nilpotent Jordan algebras of dimension $4$}

In this subsection we construct all nonisomorphic nilpotent algebras of
dimension four :

\begin{itemize}
\item Over algebraic closed field $K$ and characteristic $K\neq 2.$

\item Over real field $%
%TCIMACRO{\U{211d} }%
%BeginExpansion
\mathbb{R}
%EndExpansion
.$
\end{itemize}

Here we denote the $j$-th algebra of dimension $i$ over $K$ and $%
%TCIMACRO{\U{211d} }%
%BeginExpansion
\mathbb{R}
%EndExpansion
$\ by $J_{i,j}(K)$ and $J_{i,j}(%
%TCIMACRO{\U{211d} }%
%BeginExpansion
\mathbb{R}
%EndExpansion
)$ respectively.

\subsubsection{Nilpotent Jordan algebras with Centeral component}

\begin{enumerate}
\item \textbf{1- dimensional Trivial Central extension of }$J_{3,1}$\textbf{%
\ corresponding to }$\theta =0.$

We get the Jordan algebra :%
\begin{equation*}
\begin{tabular}{|l|l|}
\hline
$J_{4,1}(K)=J_{3,1}\oplus J_{1,1}$ & All multiplications are zero. \\ \hline
\end{tabular}%
\end{equation*}%
\begin{equation*}
\begin{tabular}{|l|l|}
\hline
$J_{4,1}(%
%TCIMACRO{\U{211d} }%
%BeginExpansion
\mathbb{R}
%EndExpansion
)=J_{3,1}\oplus J_{1,1}$ & All multiplications are zero. \\ \hline
\end{tabular}%
\end{equation*}

\item \textbf{1- dimensional Trivial Central extension of }$J_{3,2}$\textbf{%
\ corresponding to }$\theta =0.$

We get the Jordan algebra :%
\begin{equation*}
\begin{tabular}{|l|l|}
\hline
$J_{4,2}(K)=J_{3,2}\oplus J_{1,1}$ & $a^{2}=b$. \\ \hline
\end{tabular}%
\end{equation*}%
\begin{equation*}
\begin{tabular}{|l|l|}
\hline
$J_{4,2}(%
%TCIMACRO{\U{211d} }%
%BeginExpansion
\mathbb{R}
%EndExpansion
)=J_{3,2}\oplus J_{1,1}$ & $a^{2}=b$. \\ \hline
\end{tabular}%
\end{equation*}

\item \textbf{1- dimensional Trivial Central extension of }$J_{3,3}^{\alpha
} $\textbf{,}$\alpha \in K^{\ast }\diagup \left( K^{\ast }\right) ^{2},$%
\textbf{corresponding to }$\theta =0.$

We get the Jordan algebra :%
\begin{equation*}
\begin{tabular}{|l|l|}
\hline
$J_{4,3}(K)=J_{3,3}\oplus J_{1,1}$ & $a^{2}=c$ , $b^{2}=c$ $.$ \\ \hline
\end{tabular}%
\end{equation*}%
\begin{equation*}
\begin{tabular}{|l|l|}
\hline
$J_{4,3}^{\alpha =\pm 1}(%
%TCIMACRO{\U{211d} }%
%BeginExpansion
\mathbb{R}
%EndExpansion
)=J_{3,3}^{\alpha =\pm 1}\oplus J_{1,1}$ & $a^{2}=c$ , $b^{2}=\alpha c$ $.$
\ \ $\left( \alpha =\pm 1\right) $ \\ \hline
\end{tabular}%
\end{equation*}

\item \textbf{1- dimensional Trivial Central extension of }$J_{3,4}$\textbf{%
\ corresponding to }$\theta =0.$

We get the Jordan algebra :%
\begin{equation*}
\begin{tabular}{|l|l|}
\hline
$J_{4,4}(K)=J_{3,4}\oplus J_{1,1}$ & $a^{2}=b$ , $a\circ b=c$. \\ \hline
\end{tabular}%
\end{equation*}%
\begin{equation*}
\begin{tabular}{|l|l|}
\hline
$J_{4,4}(%
%TCIMACRO{\U{211d} }%
%BeginExpansion
\mathbb{R}
%EndExpansion
)=J_{3,4}\oplus J_{1,1}$ & $a^{2}=b$ , $a\circ b=c$. \\ \hline
\end{tabular}%
\end{equation*}
\end{enumerate}

\subsubsection{Nilpotent Jordan algebras without Centeral component}

\begin{enumerate}
\item 1\textbf{- dimensional Central extension of }$J_{3,1}$

Here we get that $Z^{2}(J_{3,1},K)$ is spanned by $\sum_{a,a},\sum_{b,b},%
\sum_{c,c},\sum_{a,b},\sum_{a,c}$ and $\sum_{b,c},$ also we get that $\delta
C^{1}(J_{3,1},K)=0$. Then $H^{2}(J_{3,1},K)$ consists of $\theta =\alpha
_{1}\sum_{a,a}+\alpha _{2}\sum_{b,b}+\alpha _{3}\sum_{c,c}+\alpha
_{4}\sum_{a,b}+\alpha _{5}\sum_{a,c}+\alpha _{6}\sum_{b,c}.$ The centre of $%
J_{3,1},Z(J_{3,1}),$ spanned by $a,b$ and $c$, it follows that $\theta
:J_{3,1}\times J_{3,1}\longrightarrow K$ such that%
\begin{equation*}
\theta ^{\perp }\cap Z(J_{3,1})=0
\end{equation*}%
is nondegenerate since $\theta ^{\perp }=0.$ The autommorphism group,$%
Aut(J_{3,1}),$ consists of $:$%
\begin{equation*}
\phi =\left( 
\begin{array}{ccc}
a_{11} & a_{12} & a_{13} \\ 
a_{21} & a_{22} & a_{23} \\ 
a_{31} & a_{32} & a_{33}%
\end{array}%
\right) ,\det \phi \neq 0.
\end{equation*}

\underline{Over algebraic closed field $K$ and characteristic $K\neq 2:$}

Up to equivalence there is only one nondegenerate symmetric bilinear form, $%
\sum_{a,a}+\sum_{b,b}+\sum_{c,c}.$ So there is only one $Aut(J_{3,1})$-orbit 
$\theta =\sum_{a,a}+\sum_{b,b}+\sum_{c,c},$ yielding the algebra%
\begin{equation*}
\begin{tabular}{|l|l|}
\hline
$J_{4,5}(K)$ & $a^{2}=d$ $,$ $b^{2}=d$ $,c^{2}=d$. \\ \hline
\end{tabular}%
\end{equation*}

\underline{Over real field $%
%TCIMACRO{\U{211d} }%
%BeginExpansion
\mathbb{R}
%EndExpansion
:$}

By \textit{Sylvester's law of inertia and Signature, }up to equivalence
there are four nondegenerate symmetric bilinear form

$\theta _{1}=\sum_{a,a}+\sum_{b,b}+\sum_{c,c}$

$\theta _{2}=\sum_{a,a}+\sum_{b,b}-\sum_{c,c}$

$\theta _{3}=-\sum_{a,a}-\sum_{b,b}+\sum_{c,c}$

$\theta _{4}=-\sum_{a,a}-\sum_{b,b}-\sum_{c,c}.$

We see that $\theta _{1}$ and $\theta _{4}$ are in the same $Aut(J_{3,1})$%
-orbit (since $\theta _{1}=-\theta _{4}$), also $\theta _{2}$ and $\theta
_{3}$ are in the same $Aut(J_{3,1})$-orbit (since $\theta _{2}=-\theta _{3}$%
). It remains to check that if $\theta _{1}$ and $\theta _{2}$ are in the
same $Aut(J_{3,1})$-orbit or not. We claim that $\theta _{1}$ and $\theta
_{2}$ are not in the same $Aut(J_{3,1})$-orbit. Let $\phi \in Aut(J_{3,1})$
and $\lambda \in 
%TCIMACRO{\U{211d} }%
%BeginExpansion
\mathbb{R}
%EndExpansion
^{\ast }$ such that $\phi \theta _{1}=\lambda \theta _{2}.$ Then%
\begin{eqnarray}
a_{11}^{2}+a_{21}^{2}+a_{31}^{2} &=&\lambda  \label{l1} \\
a_{12}^{2}+a_{22}^{2}+a_{32}^{2} &=&\lambda  \label{l2} \\
a_{13}^{2}+a_{23}^{2}+a_{33}^{2} &=&-\lambda  \label{l3} \\
a_{11}a_{12}+a_{21}a_{22}+a_{31}a_{32} &=&0  \notag \\
a_{11}a_{13}+a_{21}a_{23}+a_{31}a_{33} &=&0  \notag \\
a_{12}a_{13}+a_{22}a_{23}+a_{32}a_{33} &=&0  \notag
\end{eqnarray}%
Equations $\left( \ref{l1}\right) ,\left( \ref{l2}\right) $ with $\left( \ref%
{l3}\right) $ lead to $\lambda =0$ and $\phi =0.$ So $\theta _{1}$ and $%
\theta _{2}$ are not in the same $Aut(J_{3,1})$-orbit, hence we get only two
algebras 
\begin{equation*}
\begin{tabular}{|l|l|}
\hline
$J_{4,5}^{\alpha =\pm 1}(%
%TCIMACRO{\U{211d} }%
%BeginExpansion
\mathbb{R}
%EndExpansion
)$ & $a^{2}=d$ $,$ $b^{2}=d$ $,c^{2}=\alpha d$. \\ \hline
\end{tabular}%
\end{equation*}

\item 1\textbf{- dimensional Central extension of }$J_{3,2}$

Here we get that $Z^{2}(J_{3,2},K)$ is spanned by $\sum_{a,a},\sum_{a,b},%
\sum_{a,c},\sum_{b,c}$ and $\sum_{c,c}.$ Moreover, $\delta C^{1}(J_{3,2},K)$
is spanned by $\sum_{a,a}$. Then $H^{2}(J_{3,2},K)$ consists of $\theta
:=\alpha _{1}\sum_{a,b}+\alpha _{2}\sum_{a,c}+\alpha _{3}\sum_{b,c}+\alpha
_{4}\sum_{c,c}.$ The centre of $J_{3,2},Z(J_{3,2}),$ is spanned by $b$ and $%
c.$ Furthermore the automorphism group,$Aut(J_{3,2}),$ consists of :%
\begin{equation*}
\phi =%
\begin{pmatrix}
a_{11} & 0 & 0 \\ 
a_{21} & a_{11}^{2} & a_{23} \\ 
a_{31} & 0 & a_{33}%
\end{pmatrix}%
\text{ , }\det \phi \neq 0.
\end{equation*}%
The automorphism group acts as follows :%
\begin{eqnarray*}
\alpha _{1} &\longrightarrow &a_{11}^{3}\alpha _{1}+a_{31}a_{11}^{2}\alpha
_{3} \\
\alpha _{2} &\longrightarrow &a_{11}a_{23}\alpha _{1}+a_{11}a_{33}\alpha
_{2}+(a_{21}a_{33}+a_{31}a_{23})\alpha _{3}+a_{31}a_{33}\alpha _{4} \\
\alpha _{3} &\longrightarrow &a_{11}^{2}a_{33}\alpha _{3} \\
\alpha _{4} &\longrightarrow &2a_{23}a_{33}\alpha _{3}+a_{33}^{2}\alpha _{4}.
\end{eqnarray*}

We distinguish two cases.

\textbf{Case 1} \textbf{: }First suppose that $\alpha _{3}\neq 0,$ then we
can divide to get $\alpha _{3}=1.$ So we may assume that $\alpha _{3}=1,$
choose $a_{11}=a_{33}=1$, this leads to :%
\begin{eqnarray*}
\alpha _{1} &\longrightarrow &\alpha _{1}+a_{31} \\
\alpha _{2} &\longrightarrow &a_{23}\alpha _{1}+\alpha
_{2}+(a_{21}+a_{31}a_{23})+a_{31}\alpha _{4} \\
\alpha _{3} &\longrightarrow &1 \\
\alpha _{4} &\longrightarrow &2a_{23}+\alpha _{4}.
\end{eqnarray*}%
By taking $a_{31}=-\alpha _{1}$ and $a_{23}=-\frac{1}{2}\alpha _{4},$ we get 
$\alpha _{1}\longrightarrow 0$ and $\alpha _{4}\longrightarrow 0.$ Then we
can assume that $\alpha _{1}=\alpha _{4}=0$ and $\alpha _{3}=1.$ To conserve
this we set \ $a_{31}=a_{23}=0$ and $a_{11}=a_{33}=1,$ in this case we get%
\begin{eqnarray*}
\alpha _{1} &\longrightarrow &0 \\
\alpha _{2} &\longrightarrow &\alpha _{2}+a_{21} \\
\alpha _{3} &\longrightarrow &1 \\
\alpha _{4} &\longrightarrow &0.
\end{eqnarray*}%
By taking $a_{21}=-\alpha _{2},$ then $\alpha _{2}\longrightarrow 0.$ Then
we may assume that $\theta =\sum_{b,c},$ so we get the algebra :%
\begin{equation*}
\begin{tabular}{|l|l|}
\hline
$J_{4,6}(K)$ & $a^{2}=b\text{ },\text{ }b\circ c=d.$ \\ \hline
\end{tabular}%
\end{equation*}%
\begin{equation*}
\begin{tabular}{|l|l|}
\hline
$J_{4,6}(%
%TCIMACRO{\U{211d} }%
%BeginExpansion
\mathbb{R}
%EndExpansion
)$ & $a^{2}=b\text{ },\text{ }b\circ c=d.$ \\ \hline
\end{tabular}%
\end{equation*}

\textbf{Case 2} \textbf{: }If $\alpha _{3}=0,$ this ledas to%
\begin{eqnarray*}
\alpha _{1} &\longrightarrow &a_{11}^{3}\alpha _{1} \\
\alpha _{2} &\longrightarrow &a_{11}a_{23}\alpha _{1}+a_{11}a_{33}\alpha
_{2}+a_{31}a_{33}\alpha _{4} \\
\alpha _{3} &\longrightarrow &0 \\
\alpha _{4} &\longrightarrow &a_{33}^{2}\alpha _{4}.
\end{eqnarray*}

In order to have $\theta ^{\perp }\cap Z(J_{3,2})=0,$ we need $\alpha
_{1}\neq 0$ and $(\alpha _{2},\alpha _{4})\neq (0,0).$ So after dividing we
may assume $\alpha _{1}=1.$ Choose $a_{11}=a_{33}=1,a_{31}=0$ and $%
a_{23}=-\alpha _{2},$ this leads to :%
\begin{eqnarray*}
\alpha _{1} &\longrightarrow &1 \\
\alpha _{2} &\longrightarrow &0 \\
\alpha _{3} &\longrightarrow &0 \\
\alpha _{4} &\longrightarrow &\alpha _{4}\neq 0.
\end{eqnarray*}%
then we get the cocycle $\theta =\sum_{a,b}+\alpha _{4}\sum_{c,c}.$ Now the
problem is to describe exactly when \ there exist $\phi \in Aut(J_{3,2})$
with $\phi (\sum_{a,b}+\alpha _{4}\sum_{c,c})=\lambda (\sum_{a,b}+\alpha
_{4}^{^{\prime }}\sum_{c,c})$ for some $\lambda \in K^{\ast }.$ Then we get $%
\phi (\sum_{a,b}+\alpha _{4}\sum_{c,c})=\lambda (\sum_{a,b}+\alpha
_{4}^{^{\prime }}\sum_{c,c})$ if and only if there exist $a_{11},a_{33}\in
K^{\ast }$ such that $a_{33}^{2}\alpha _{4}=a_{11}^{3}\alpha _{4}^{^{\prime
}}.$ Then there is only one $Aut(J_{3,2})$-orbit, with representative $%
\theta =\sum_{a,b}+\sum_{c,c}.$ So we get the algebra 
\begin{equation*}
\begin{tabular}{|l|l|}
\hline
$J_{4,7}(K)$ & $a^{2}=b\text{ },\text{ }a\circ b=d\text{ },\text{ }c^{2}=d.$
\\ \hline
\end{tabular}%
\end{equation*}%
\begin{equation*}
\begin{tabular}{|l|l|}
\hline
$J_{4,7}(%
%TCIMACRO{\U{211d} }%
%BeginExpansion
\mathbb{R}
%EndExpansion
)$ & $a^{2}=b\text{ },\text{ }a\circ b=d\text{ },\text{ }c^{2}=d.$ \\ \hline
\end{tabular}%
\end{equation*}

\textit{Remarks : }

\begin{itemize}
\item \textit{For any }$\alpha ,\beta \in K^{\ast }$ there exist $x,y\in
K^{\ast }$ such that $x^{3}\alpha =y^{2}\beta ,$ by taking $x=\alpha \beta $
and $y=\alpha ^{2}\beta .$

\item Let $J_{4,7}^{\alpha }:a^{2}=b$ $,$ $a\circ b=d$ $,$ $c^{2}=\alpha d.$
Now putting $a^{^{\prime }}=\alpha a,b^{^{\prime }}=\alpha ^{2}b,c^{^{\prime
}}=\gamma c,d^{^{\prime }}=\alpha ^{3}d$ we get the same multiplication
table, but the parameter has changed to $1$.So we may assume $\alpha =1,$
and we get only one algebra $J_{4,7}$.

\item Cocycles $\sum_{a,b}+\sum_{c,c}$ and $\sum_{b,c}$ not lie in the same $%
Aut(J_{3,2})$-orbit, $\phi \left( \sum_{a,b}+\sum_{c,c}\right)
=a_{11}^{3}\sum_{a,b}+(a_{11}a_{23}+a_{31}a_{33})\sum_{a,c}+a_{33}^{2}%
\sum_{c,c}.$ Clearly $J_{4,6}$ and $J_{4,7}$ are not isomorphic, $J_{4,6}$
not associative but $J_{4,7}$ is a commutative nilpotent associative algebra.
\end{itemize}

\item 1\textbf{- dimensional Central extension of }$J_{3,3}^{\alpha },\alpha
\in K^{\ast }\diagup \left( K^{\ast }\right) ^{2}$

Here we get that $Z^{2}(J_{3,3}^{\alpha },K)$ consists of :

$\theta :=\alpha _{1}\sum_{a,a}+\alpha _{2}\sum_{b,b}+\alpha
_{3}\sum_{a,b}+\alpha _{4}\sum_{a,c}+\alpha _{5}\sum_{b,c}.$

Moreover, $\delta C^{1}(J_{3,3}^{\alpha },K)$ is spanned by $%
\sum_{a,a}+\alpha \sum_{b,b},$ then $H^{2}(J_{3,3}^{\alpha },K)$ is consists
of $\theta =\alpha _{1}\sum_{a,a}+\alpha _{2}\sum_{b,b}+\alpha
_{3}\sum_{a,b}+\alpha _{4}\sum_{a,c}+\alpha _{5}\sum_{b,c}$ modulo $%
\sum_{a,a}+\alpha \sum_{b,b}.$ Furthermore the automorphism group,$%
Aut(J_{3,3}^{\alpha })$, consists of :%
\begin{equation*}
\phi =\left( 
\begin{array}{ccc}
a_{11} & a_{12} & 0 \\ 
a_{21} & a_{22} & 0 \\ 
a_{31} & a_{32} & a_{33}%
\end{array}%
\right)
\end{equation*}%
such that%
\begin{eqnarray*}
a_{11}^{2}+\alpha a_{21}^{2} &=&a_{33}\text{ },\text{ }a_{12}^{2}+\alpha
a_{22}^{2}=a_{33} \\
a_{11}a_{12}+\alpha a_{21}a_{22} &=&0\text{ },\text{ }a_{33}\delta
=a_{33}(a_{11}a_{22}-a_{21}a_{12})\neq 0.
\end{eqnarray*}%
Write $\phi \theta =\alpha _{1}^{\prime }\sum_{a,a}+\alpha _{2}^{\prime
}\sum_{b,b}+\alpha _{3}^{\prime }\sum_{a,b}+\alpha _{4}^{\prime
}\sum_{a,c}+\alpha _{5}^{\prime }\sum_{b,c}.$ Then%
\begin{eqnarray*}
\alpha _{1}^{^{\prime }} &=&a_{11}^{2}\alpha _{1}+a_{21}^{2}\alpha
_{2}+2a_{11}a_{21}\alpha _{2}+a_{11}a_{31}\alpha _{4}+a_{21}a_{31}\alpha _{5}
\\
\alpha _{2}^{^{\prime }} &=&a_{12}^{2}\alpha _{1}+a_{22}^{2}\alpha
_{2}+2a_{11}a_{21}\alpha _{2}+a_{12}a_{32}\alpha _{4}+a_{22}a_{32}\alpha _{5}
\\
\alpha _{3}^{^{\prime }} &=&a_{11}a_{12}\alpha _{1}+a_{22}a_{21}\alpha
_{2}+(a_{11}a_{22}+a_{21}a_{12})\alpha _{3} \\
&&+(a_{11}a_{32}+a_{31}a_{12})\alpha _{4}+(a_{21}a_{32}+a_{31}a_{22})\alpha
_{5} \\
\alpha _{4}^{^{\prime }} &=&\delta \left( a_{22}\alpha \alpha
_{3}-a_{12}\alpha _{4}\right) \\
\alpha _{5}^{^{\prime }} &=&a_{33}\left( a_{12}\alpha _{3}+a_{22}\alpha
_{4}\right) .
\end{eqnarray*}

The centre of $J_{3,3}^{\alpha },Z(J_{3,3}^{\alpha }),$ is spanned by $c$,
so in order to have $\theta ^{\perp }\cap Z(J_{3,3}^{\alpha })=0$ we need
one of $\alpha _{4},\alpha _{5}$ nonzero. Without loss of generality, we can
assume that $\alpha _{4}=1$. Then

$\theta :=\alpha _{1}\sum_{a,a}+\alpha _{2}\sum_{b,b}+\alpha
_{3}\sum_{a,b}+\sum_{a,c}+\alpha _{5}\sum_{b,c}$ modulo $\sum_{a,a}+\alpha
\sum_{b,b}.$

\underline{\underline{\textbf{Case }$\alpha =1$ :}}

We get 
\begin{eqnarray*}
\alpha _{1}^{^{\prime }} &=&a_{11}^{2}\alpha _{1}+a_{21}^{2}\alpha
_{2}+2a_{11}a_{21}\alpha _{2}+a_{11}a_{31}+a_{21}a_{31}\alpha _{5} \\
\alpha _{2}^{^{\prime }} &=&a_{12}^{2}\alpha _{1}+a_{22}^{2}\alpha
_{2}+2a_{11}a_{21}\alpha _{2}+a_{12}a_{32}+a_{22}a_{32}\alpha _{5} \\
\alpha _{3}^{^{\prime }} &=&a_{11}a_{12}\alpha _{1}+a_{22}a_{21}\alpha
_{2}+(a_{11}a_{22}+a_{21}a_{12})\alpha _{3} \\
&&+a_{11}a_{32}+a_{31}a_{12}+(a_{21}a_{32}+a_{31}a_{22})\alpha _{5} \\
\alpha _{4}^{^{\prime }} &=&\delta \left( a_{22}-a_{12}\alpha _{5}\right) \\
\alpha _{5}^{^{\prime }} &=&a_{33}\left( a_{12}+a_{22}\alpha _{5}\right) .
\end{eqnarray*}%
\underline{If $1+\alpha _{5}^{2}\neq 0:$}

We can choose $a_{22}=\frac{1}{\delta \left( 1+\alpha _{5}^{2}\right) }$ and 
$a_{12}=\frac{-\alpha _{5}}{\delta \left( 1+\alpha _{5}^{2}\right) }$ to get 
$\alpha _{4}^{^{\prime }}=1$ and $\alpha _{5}^{^{\prime }}=0.$ So may assume
that $\alpha _{4}=1$ and $\alpha _{5}=0,$ to fix $\theta $, choose $%
a_{12}=a_{21}=0$ and $a_{11}=a_{22}=1.$ It follows that 
\begin{eqnarray*}
\alpha _{1}^{^{\prime }} &=&\alpha _{1}+a_{31} \\
\alpha _{2}^{^{\prime }} &=&\alpha _{2} \\
\alpha _{3}^{^{\prime }} &=&\alpha _{3}+a_{32} \\
\alpha _{4}^{^{\prime }} &=&1 \\
\alpha _{5}^{^{\prime }} &=&0.
\end{eqnarray*}%
Choose $a_{31}=\alpha _{2}-\alpha _{1}$ and $a_{32}=-\alpha _{3}.$

So we can assume that $\theta =\alpha _{2}\left(
\sum_{a,a}+\sum_{b,b}\right) +\sum_{a,c}$ modulo $\sum_{a,a}+\sum_{b,b},$
then $\theta =\sum_{a,c}.$ We get algebra%
\begin{equation*}
\begin{tabular}{|l|l|}
\hline
$J_{4,8}(K)$ & $a^{2}=c\text{ },\text{ }b^{2}=c\text{ },\text{ }a\circ c=d.$
\\ \hline
\end{tabular}%
\end{equation*}%
\begin{equation*}
\begin{tabular}{|l|l|}
\hline
$J_{4,8}(%
%TCIMACRO{\U{211d} }%
%BeginExpansion
\mathbb{R}
%EndExpansion
)$ & $a^{2}=c\text{ },\text{ }b^{2}=c\text{ },\text{ }a\circ c=d.$ \\ \hline
\end{tabular}%
\end{equation*}%
\underline{If $1+\alpha _{5}^{2}=0:$}

\underline{Over algebraic closed field $K$ and characteristic $K\neq 2:$}

We have 
\begin{eqnarray*}
\alpha _{1}^{^{\prime }} &=&a_{11}^{2}\alpha _{1}+a_{21}^{2}\alpha
_{2}+2a_{11}a_{21}\alpha _{2}+a_{11}a_{31}+a_{21}a_{31}\alpha _{5} \\
\alpha _{2}^{^{\prime }} &=&a_{12}^{2}\alpha _{1}+a_{22}^{2}\alpha
_{2}+2a_{11}a_{21}\alpha _{2}+a_{12}a_{32}+a_{22}a_{32}\alpha _{5} \\
\alpha _{3}^{^{\prime }} &=&a_{11}a_{12}\alpha _{1}+a_{22}a_{21}\alpha
_{2}+(a_{11}a_{22}+a_{21}a_{12})\alpha _{3} \\
&&+a_{11}a_{32}+a_{31}a_{12}+(a_{21}a_{32}+a_{31}a_{22})\alpha _{5} \\
\alpha _{4}^{^{\prime }} &=&\delta \left( a_{22}-a_{12}\alpha _{5}\right) \\
\alpha _{5}^{^{\prime }} &=&a_{33}\left( a_{12}+a_{22}\alpha _{5}\right) .
\end{eqnarray*}%
Choose $a_{12}=a_{21}=0,a_{11}=a_{22}=1,a_{31}=-\alpha _{1}$ and $a_{32}=%
\frac{-\alpha _{2}}{\alpha _{5}}$ (note that $\alpha _{5}\neq 0$ since $%
1+\alpha _{5}^{2}=0$)$.$ It follows that $\alpha _{1}^{^{\prime }}=\alpha
_{2}^{^{\prime }}=0,\alpha _{4}^{^{\prime }}=1,\alpha _{5}^{^{\prime
}}=\alpha _{5}$ and $1+\alpha _{5}^{2}=1+\alpha _{5}^{\prime 2}=0.$

So we assume that $\theta :=\alpha _{3}\sum_{a,b}+\sum_{a,c}+\alpha
_{5}\sum_{b,c}$ such that $1+\alpha _{5}^{2}=0.$ If $\alpha _{3}=0$ we get
the algebra :%
\begin{equation*}
\begin{tabular}{|l|}
\hline
$a^{2}=c\text{ },\text{ }b^{2}=c\text{ },\text{ }a\circ c=d,b\circ c=\alpha
_{5}d,1+\alpha _{5}^{2}=0.$ \\ \hline
\end{tabular}%
\end{equation*}%
The change of basis $a\longrightarrow a,b\longrightarrow -\alpha
_{5}b,c\longrightarrow c$ and $d\longrightarrow d,$ yields to%
\begin{equation*}
\begin{tabular}{|l|l|}
\hline
$J_{4,9}(K)$ & $a^{2}=c\text{ },\text{ }b^{2}=-c\text{ },\text{ }a\circ c=d$ 
$,$ $b\circ c=d$ $.$ \\ \hline
\end{tabular}%
\end{equation*}%
If $\alpha _{3}\neq 0$ , by consider the previous change of the basis we get
the algebra :%
\begin{equation*}
\begin{tabular}{|l|}
\hline
$a^{2}=c\text{ },\text{ }b^{2}=-c\text{ },\text{ }a\circ c=d$ $,b\circ c=d$ $%
,$ $a\circ b=\alpha _{3}d.$ \\ \hline
\end{tabular}%
\end{equation*}%
Now putting $a^{\prime }=\alpha _{3}a,b^{\prime }=\alpha _{3}b,c^{\prime
}=\alpha _{3}^{2}c,d^{\prime }=\alpha _{3}^{3}d$ we get the same
multiplication table, but the parameter has changed to 1. So we may assume $%
\alpha _{3}=1$, and we get only algebra%
\begin{equation*}
\begin{tabular}{|l|l|}
\hline
$J_{4,10}(K)$ & $a^{2}=c\text{ },\text{ }b^{2}=-c\text{ },\text{ }a\circ c=d$
$,$ $b\circ c=d$ $,$ $a\circ b=d.$ \\ \hline
\end{tabular}%
\end{equation*}

\underline{Over real field $%
%TCIMACRO{\U{211d} }%
%BeginExpansion
\mathbb{R}
%EndExpansion
:$}

There is no $\alpha _{5}\in 
%TCIMACRO{\U{211d} }%
%BeginExpansion
\mathbb{R}
%EndExpansion
$ such that $1+\alpha _{5}^{2}=0.$

We are completed 1\textbf{- d}imensional Central extension of $%
J_{3,3}^{\alpha }$ $\alpha \in K^{\ast }\diagup \left( K^{\ast }\right) ^{2}$%
over algebraic closed field $K$ and characteristic $K\neq 2$ since $K^{\ast
}\diagup \left( K^{\ast }\right) ^{2}=\{1\}.$ To complete classification
over $%
%TCIMACRO{\U{211d} }%
%BeginExpansion
\mathbb{R}
%EndExpansion
,$ it remains to consider case $\alpha =-1.$

\underline{\underline{\textbf{Case }$\alpha =-1$ :}}

We get : 
\begin{eqnarray*}
\alpha _{1}^{^{\prime }} &=&a_{11}^{2}\alpha _{1}+a_{21}^{2}\alpha
_{2}+2a_{11}a_{21}\alpha _{2}+a_{11}a_{31}+a_{21}a_{31}\alpha _{5} \\
\alpha _{2}^{^{\prime }} &=&a_{12}^{2}\alpha _{1}+a_{22}^{2}\alpha
_{2}+2a_{11}a_{21}\alpha _{2}+a_{12}a_{32}+a_{22}a_{32}\alpha _{5} \\
\alpha _{3}^{^{\prime }} &=&a_{11}a_{12}\alpha _{1}+a_{22}a_{21}\alpha
_{2}+(a_{11}a_{22}+a_{21}a_{12})\alpha _{3} \\
&&+a_{11}a_{32}+a_{31}a_{12}+(a_{21}a_{32}+a_{31}a_{22})\alpha _{5} \\
\alpha _{4}^{^{\prime }} &=&-\delta \left( a_{22}+a_{12}\alpha _{5}\right) \\
\alpha _{5}^{^{\prime }} &=&a_{33}\left( a_{12}+a_{22}\alpha _{5}\right) .
\end{eqnarray*}%
As we do before

\underline{If $1-\alpha _{5}^{2}\neq 0:$}

We can choose $a_{22}=\frac{-1}{\delta \left( 1-\alpha _{5}^{2}\right) }$
and $a_{12}=\frac{\alpha 5}{\delta \left( 1-\alpha _{5}^{2}\right) }$ to get 
$\alpha _{4}^{^{\prime }}=1$ and $\alpha _{5}^{^{\prime }}=0.$ So may assume
that $\alpha _{4}=1$ and $\alpha _{5}=0,$ to fix $\theta $, choose $%
a_{12}=a_{21}=0$ and $a_{11}=a_{22}=1.$ It follows that 
\begin{eqnarray*}
\alpha _{1}^{^{\prime }} &=&\alpha _{1}+a_{31} \\
\alpha _{2}^{^{\prime }} &=&\alpha _{2} \\
\alpha _{3}^{^{\prime }} &=&\alpha _{3}+a_{32} \\
\alpha _{4}^{^{\prime }} &=&1 \\
\alpha _{5}^{^{\prime }} &=&0.
\end{eqnarray*}%
Choose $a_{31}=-\alpha _{2}-\alpha _{1}$ and $a_{32}=-\alpha _{3}.$

So we can assume that $\theta =-\alpha _{2}\left(
\sum_{a,a}-\sum_{b,b}\right) +\sum_{a,c}$ modulo $\sum_{a,a}-\sum_{b,b},$
then $\theta =\sum_{a,c}.$ We get algebra%
\begin{equation*}
\begin{tabular}{|l|l|}
\hline
$J_{4,9}(%
%TCIMACRO{\U{211d} }%
%BeginExpansion
\mathbb{R}
%EndExpansion
)$ & $a^{2}=c\text{ },\text{ }b^{2}=-c\text{ },\text{ }a\circ c=d.$ \\ \hline
\end{tabular}%
\end{equation*}%
\underline{If $1-\alpha _{5}^{2}=0:$}

We have 
\begin{eqnarray*}
\alpha _{1}^{^{\prime }} &=&a_{11}^{2}\alpha _{1}+a_{21}^{2}\alpha
_{2}+2a_{11}a_{21}\alpha _{2}+a_{11}a_{31}+a_{21}a_{31}\alpha _{5} \\
\alpha _{2}^{^{\prime }} &=&a_{12}^{2}\alpha _{1}+a_{22}^{2}\alpha
_{2}+2a_{11}a_{21}\alpha _{2}+a_{12}a_{32}+a_{22}a_{32}\alpha _{5} \\
\alpha _{3}^{^{\prime }} &=&a_{11}a_{12}\alpha _{1}+a_{22}a_{21}\alpha
_{2}+(a_{11}a_{22}+a_{21}a_{12})\alpha _{3} \\
&&+a_{11}a_{32}+a_{31}a_{12}+(a_{21}a_{32}+a_{31}a_{22})\alpha _{5} \\
\alpha _{4}^{^{\prime }} &=&-\delta \left( a_{22}+a_{12}\alpha _{5}\right) \\
\alpha _{5}^{^{\prime }} &=&a_{33}\left( a_{12}+a_{22}\alpha _{5}\right) .
\end{eqnarray*}%
Choose $a_{12}=a_{21}=0,a_{11}=a_{22}=1,a_{31}=-\alpha _{1}$ and $a_{32}=%
\frac{-\alpha _{2}}{\alpha _{5}}$ (note that $\alpha _{5}\neq 0$ since $%
1-\alpha _{5}^{2}=0$)$.$ It follows that $\alpha _{1}^{^{\prime }}=\alpha
_{2}^{^{\prime }}=0,\alpha _{4}^{^{\prime }}=1,\alpha _{5}^{^{\prime
}}=\alpha _{5}$ and $1-\alpha _{5}^{2}=1-\alpha _{5}^{\prime 2}=0.$

So we assume that $\theta :=\alpha _{3}\sum_{a,b}+\sum_{a,c}+\alpha
_{5}\sum_{b,c}$ such that $1-\alpha _{5}^{2}=0.$ If $\alpha _{3}=0$ we get
the algebra :%
\begin{equation*}
\begin{tabular}{|l|}
\hline
$a^{2}=c\text{ },\text{ }b^{2}=-c\text{ },\text{ }a\circ c=d,b\circ c=\alpha
_{5}d,1-\alpha _{5}^{2}=0.$ \\ \hline
\end{tabular}%
\end{equation*}%
The change of basis $a\longrightarrow a,b\longrightarrow \alpha
_{5}b,c\longrightarrow c$ and $d\longrightarrow d,$ yields to%
\begin{equation*}
\begin{tabular}{|l|l|}
\hline
$J_{4,10}(%
%TCIMACRO{\U{211d} }%
%BeginExpansion
\mathbb{R}
%EndExpansion
)$ & $a^{2}=c\text{ },\text{ }b^{2}=-c\text{ },\text{ }a\circ c=d$ $,$ $%
b\circ c=d$ $.$ \\ \hline
\end{tabular}%
\end{equation*}%
If $\alpha _{3}\neq 0$ , by consider the previous change of the basis we get
the algebra :%
\begin{equation*}
\begin{tabular}{|l|}
\hline
$a^{2}=c\text{ },\text{ }b^{2}=-c\text{ },\text{ }a\circ c=d$ $,b\circ c=d$ $%
,$ $a\circ b=\alpha _{3}d.$ \\ \hline
\end{tabular}%
\end{equation*}%
Now putting $a^{\prime }=\alpha _{3}a,b^{\prime }=\alpha _{3}b,c^{\prime
}=\alpha _{3}^{2}c,d^{\prime }=\alpha _{3}^{3}d$ we get the same
multiplication table, but the parameter has changed to 1. So we may assume $%
\alpha _{3}=1$, and we get only algebra%
\begin{equation*}
\begin{tabular}{|l|l|}
\hline
$J_{4,11}(%
%TCIMACRO{\U{211d} }%
%BeginExpansion
\mathbb{R}
%EndExpansion
)$ & $a^{2}=c\text{ },\text{ }b^{2}=-c\text{ },\text{ }a\circ c=d$ $,$ $%
b\circ c=d$ $,$ $a\circ b=d.$ \\ \hline
\end{tabular}%
\end{equation*}

\textbf{Remark} :

One can consider the algebra $J:a\circ b=c$ instead of $J_{3,3}^{\alpha }$
over algebaic closed field $K$ and $ch(K)\neq 2$, by change of basis $%
a\longrightarrow a+\sigma b,b\longrightarrow a-\sigma b,c\longrightarrow 2c$
and $\sigma ^{2}+1=0.$ Also for $J_{3,3}^{\alpha },\alpha =-1$ in the case
of the Real field $%
%TCIMACRO{\U{211d} }%
%BeginExpansion
\mathbb{R}
%EndExpansion
,$ by change of basis $a\longrightarrow a+b,b\longrightarrow
a-b,c\longrightarrow 2c.$ Calculations by this consideration may by some
simple. Here we get that $H^{2}(J,K)$ consists of $\theta =\alpha
_{1}\sum_{a,a}+\alpha _{2}\sum_{b,b}+\alpha _{3}\sum_{a,c}+\alpha
_{4}\sum_{b,c}$ , and the automorphism group consists of 
\begin{equation*}
\phi =\left( 
\begin{array}{ccc}
a_{11} & a_{12} & 0 \\ 
a_{21} & a_{22} & 0 \\ 
a_{31} & a_{32} & a_{11}a_{22}+a_{21}a_{12}%
\end{array}%
\right) ,a_{11}^{2}a_{22}^{2}-a_{12}^{2}a_{21}^{2}\neq 0\text{ and \ }%
a_{11}a_{21}=a_{12}a_{22}=0\text{ }
\end{equation*}%
Then$\ \ \phi \theta =\alpha _{1}^{^{\prime }}\sum_{a,a}+\alpha
_{2}^{^{\prime }}\sum_{b,b}+\alpha _{3}^{^{\prime }}\sum_{a,c}+\alpha
_{4}^{^{\prime }}\sum_{b,c}$%
\begin{eqnarray*}
\alpha _{1}^{^{\prime }} &=&a_{11}^{2}\alpha _{1}+a_{21}^{2}\alpha
_{2}+2a_{11}a_{31}\alpha _{3}+2a_{21}a_{31}\alpha _{4} \\
\alpha _{2}^{^{\prime }} &=&a_{12}^{2}\alpha _{1}+a_{22}^{2}\alpha
_{2}+2a_{12}a_{32\alpha _{3}}+2a_{22}a_{32}\alpha _{4} \\
\alpha _{3}^{^{\prime }} &=&\left( a_{11}a_{22}+a_{21}a_{12}\right) \left(
a_{11}\alpha _{3}+a_{21}\alpha _{4}\right) \\
\alpha _{4}^{^{\prime }} &=&\left( a_{11}a_{22}+a_{21}a_{12}\right) \left(
a_{12}\alpha _{3}+a_{22}\alpha _{4}\right)
\end{eqnarray*}%
We find that there exist only three orbits, $\theta _{1}=\sum_{a,c},\theta
_{2}=\sum_{a,c}+\sum_{b,b}$ and $\theta _{3}=\sum_{a,c}+\sum_{b,c},$ which
be the same number of orbits we get for these algebras.

\item 1\textbf{- dimensional Central extension of }$J_{3,4}$

Here $Z^{2}(J_{3,4},K)$ is spanned by $\sum_{a,a},\sum_{a,b}$ and $%
\sum_{a,c}+\sum_{b,b}.$ Moreover, $\delta C^{1}(J_{3,4},K)$ is spanned by $%
\sum_{a,a}$ and $\sum_{a,b}$. Then $H^{2}(J_{3,4},K)$ is spanned by $\theta
:=\sum_{a,c}+\sum_{b,b},$ so we get only one algebra%
\begin{equation*}
\begin{tabular}{|l|l|}
\hline
$J_{4,11}(K)$ & $a^{2}=b\text{ },\text{ }a\circ b=c$ $,$ $a\circ c=d$ $,$ $%
b^{2}=d.$ \\ \hline
\end{tabular}%
\end{equation*}%
\begin{equation*}
\begin{tabular}{|l|l|}
\hline
$J_{4,12}(%
%TCIMACRO{\U{211d} }%
%BeginExpansion
\mathbb{R}
%EndExpansion
)$ & $a^{2}=b\text{ },\text{ }a\circ b=c$ $,$ $a\circ c=d$ $,$ $b^{2}=d.$ \\ 
\hline
\end{tabular}%
\end{equation*}

\item 2\textbf{- dimensional Central extension of }$J_{2,1}$

Here $H^{2}(J_{2,1},K)$ is spanned by $\sum_{a,a},\sum_{a,b}$ and $%
\sum_{b,b},$ then $G_{2}(H^{2}(J_{2,1},K))$ consists of :

$\theta :=\alpha _{1}$ $\sum_{a,a}\wedge \sum_{a,b}+\alpha
_{2}\sum_{a,a}\wedge \sum_{b,b}+\alpha _{3}\sum_{a,b}\wedge \sum_{b,b}.$

The autommorphism group,$Aut(J_{2,1}),$ consists of $:$%
\begin{equation*}
\phi =\left( 
\begin{array}{cc}
a_{11} & a_{12} \\ 
a_{21} & a_{22}%
\end{array}%
\right) ,\det \phi \neq 0.
\end{equation*}%
Now we need to find representatives of $U_{2}(J_{2,1})\diagup Aut(J_{2,1}),$
write

$\phi \theta :=\alpha _{1}^{\prime }$ $\sum_{a,a}\wedge \sum_{a,b}+\alpha
_{2}^{\prime }\sum_{a,a}\wedge \sum_{b,b}+\alpha _{3}^{\prime
}\sum_{a,b}\wedge \sum_{b,b}.$

Then 
\begin{eqnarray*}
\alpha _{1}^{\prime } &=&\left( a_{11}^{2}\alpha _{1}+a_{11}a_{21}\alpha
_{2}+a_{21}^{2}\alpha _{3}\right) \det \phi \\
\alpha _{2}^{\prime } &=&\left( 2a_{11}a_{12}\alpha
_{1}+(a_{11}a_{22}+a_{12}a_{21})\alpha _{2}+2a_{21}a_{22}\alpha _{3}\right)
\det \phi \\
\alpha _{3}^{\prime } &=&\left( a_{12}^{2}\alpha _{1}+a_{12}a_{22}\alpha
_{2}+a_{22}^{2}\alpha _{3}\right) \det \phi .
\end{eqnarray*}%
Without loss of generality, we can assume that $\alpha _{1}=1$ since $\left(
\alpha _{1},\alpha _{2},\alpha _{3}\right) $ $\neq \left( 0,0,0\right) .$
Choose $a_{11}=a_{22}=1,a_{21}=0,a_{12}=\frac{-\alpha _{2}}{2}$ to get $%
\alpha _{1}^{\prime }=1$ and $\alpha _{2}^{\prime }=0.$

So we may assume $\theta =$ $\sum_{a,a}\wedge \sum_{a,b}+\alpha
_{3}\sum_{a,b}\wedge \sum_{b,b}.$

\underline{If $\alpha _{3}=0:$}

Then we get the cocycle $\theta _{1}=\sum_{a,a}\wedge \sum_{a,b}.$

\underline{If $\alpha _{3}\neq 0:$}%
\begin{eqnarray*}
\alpha _{1}^{\prime } &=&\left( a_{11}^{2}\alpha _{1}+a_{21}^{2}\alpha
_{3}\right) \det \phi \\
\alpha _{2}^{\prime } &=&\left( 2a_{11}a_{12}\alpha _{1}+2a_{21}a_{22}\alpha
_{3}\right) \det \phi \\
\alpha _{3}^{\prime } &=&\left( a_{12}^{2}\alpha _{1}+a_{22}^{2}\alpha
_{3}\right) \det \phi .
\end{eqnarray*}

To fix $\alpha _{2}^{\prime }=0$ choose $a_{21}=a_{12}=0,$ then 
\begin{eqnarray*}
\alpha _{1}^{\prime } &=&\det \phi \left( a_{11}^{2}\right) \\
\alpha _{2}^{\prime } &=&0 \\
\alpha _{3}^{\prime } &=&\det \phi \left( a_{22}^{2}\alpha _{3}\right) .
\end{eqnarray*}

\underline{Over algebraic closed field $K$ and characteristic $K\neq 2:$}

In this case every element is a square, then we choose $a_{11}=1,a_{22}=%
\frac{1}{\sqrt{-\alpha _{3}}}$ and then dividing by $\det \phi .$ We get $%
\theta _{2}=$ $\sum_{a,a}\wedge \sum_{a,b}-\sum_{a,b}\wedge
\sum_{b,b}=\left( \sum_{a,a}+\sum_{b,b}\right) \wedge \sum_{a,b}.$ We claim
that $\theta _{1}$ and $\theta _{2}$ are not in the same orbit. Let $\phi
\in Aut(J_{2,1})$ and $\lambda \in K^{\ast }$ such that $\phi $ $\theta
_{1}=\lambda \theta _{2},$then%
\begin{eqnarray}
\lambda &=&\left( a_{11}^{2}\right) \det \phi  \notag \\
0 &=&\left( 2a_{11}a_{12}\right) \det \phi  \label{m} \\
-\lambda &=&\left( a_{12}^{2}\right) \det \phi .  \notag
\end{eqnarray}%
Equation $\left( \ref{m}\right) $ ensures that $\lambda =0.$ So we get only
two $Aut(J_{2,1})$-orbits over $K,$ yielding the two algebras%
\begin{equation*}
\begin{tabular}{|l|l|}
\hline
$J_{4,12}(K)$ & $a^{2}=c\text{ },\text{ }a\circ b=d$ $.$ \\ \hline
\end{tabular}%
\end{equation*}%
\begin{equation*}
\begin{tabular}{|l|l|}
\hline
$J_{4,13}(K)$ & $a^{2}=c\text{ },\text{ }b^{2}=c$ $,$ $a\circ b=d$ $.$ \\ 
\hline
\end{tabular}%
\end{equation*}

\underline{Over real field $%
%TCIMACRO{\U{211d} }%
%BeginExpansion
\mathbb{R}
%EndExpansion
:$}

If $\alpha _{3}<0,$ choose $a_{11}=1,a_{22}=\frac{1}{\sqrt{-\alpha _{3}}}$
and then dividing by $\det \phi .$We get :

$\theta _{2}=$ $\sum_{a,a}\wedge \sum_{a,b}-\sum_{a,b}\wedge
\sum_{b,b}=\left( \sum_{a,a}+\sum_{b,b}\right) \wedge \sum_{a,b}.$

If $\alpha _{3}>0,$ choose $a_{11}=1,a_{22}=\frac{1}{\sqrt{\alpha _{3}}}$
and then dividing by $\det \phi .$ We get :

$\theta _{3}=$ $\sum_{a,a}\wedge \sum_{a,b}+\sum_{a,b}\wedge
\sum_{b,b}=\left( \sum_{a,a}-\sum_{b,b}\right) \wedge \sum_{a,b}.$

Clearly $\theta _{1}$ is not conjugate to any of $\theta _{2}$ and $\theta
_{3}.$ It remains to check if $\theta _{2}$ and $\theta _{3}$ are in the
same $Aut(J_{2,1})$-orbit or not. Let $\phi \in Aut(J_{2,1})$ and $\lambda
\in 
%TCIMACRO{\U{211d} }%
%BeginExpansion
\mathbb{R}
%EndExpansion
^{\ast }$ such that $\phi $ $\theta _{3}=\lambda \theta _{2},$then%
\begin{eqnarray}
\lambda &=&\left( a_{11}^{2}+a_{21}^{2}\right) \det \phi  \label{m1} \\
0 &=&\left( 2a_{11}a_{12}+2a_{21}a_{22}\right) \det \phi  \notag \\
-\lambda &=&\left( a_{12}^{2}+a_{22}^{2}\right) \det \phi .  \label{m2}
\end{eqnarray}%
Equations $\left( \ref{m1}\right) $ and $\left( \ref{m2}\right) $ Lead to $%
\lambda =0$ and $\phi =0.$ So we get only three $Aut(J_{2,1})$-orbits over $%
%TCIMACRO{\U{211d} }%
%BeginExpansion
\mathbb{R}
%EndExpansion
,$ yielding the three algebras%
\begin{equation*}
\begin{tabular}{|l|l|}
\hline
$J_{4,13}(%
%TCIMACRO{\U{211d} }%
%BeginExpansion
\mathbb{R}
%EndExpansion
)$ & $a^{2}=c\text{ },\text{ }a\circ b=d$ $.$ \\ \hline
\end{tabular}%
\end{equation*}%
\begin{equation*}
\begin{tabular}{|l|l|}
\hline
$J_{4,14}(%
%TCIMACRO{\U{211d} }%
%BeginExpansion
\mathbb{R}
%EndExpansion
)$ & $a^{2}=c\text{ },\text{ }b^{2}=c$ $,$ $a\circ b=d$ $.$ \\ \hline
\end{tabular}%
\end{equation*}%
\begin{equation*}
\begin{tabular}{|l|l|}
\hline
$J_{4,15}(%
%TCIMACRO{\U{211d} }%
%BeginExpansion
\mathbb{R}
%EndExpansion
)$ & $a^{2}=c\text{ },\text{ }b^{2}=-c$ $,$ $a\circ b=d$ $.$ \\ \hline
\end{tabular}%
\end{equation*}

Summarize our results in the following theorems :
\end{enumerate}

\begin{theorem}
Up to isomorphism there exist $13$ nilpotent Jordan algebras of dimension $4$
over algebraic closed field $K$\ and $ch(K)\neq 2$ which are isomorphic to
one of the following \ pairwise non-isomorphic nilpotent Jordan algebras :%
\begin{equation*}
\begin{tabular}{|l|}
\hline
\textbf{Nilpotent Jordan algebras with Centeral component} \\ \hline
$J_{4,1}=J_{3,1}\oplus J_{1,1}.$ $:$All multiplications are zero. \\ \hline
$J_{4,2}=J_{3,2}\oplus J_{1,1}.$ $:$ $a^{2}=b.$ \\ \hline
$J_{4,3}=J_{3,3}\oplus J_{1,1}.$ $:$ $a^{2}=c$ $,$ $b^{2}=c.$ \\ \hline
$J_{4,4}=J_{3,4}\oplus J_{1,1}.$ $:$ $a^{2}=b$ $,$ $a\circ b=c.$ \\ \hline
\textbf{Nilpotent Jordan algebras without Centeral component} \\ \hline
$J_{4,5}$ $:$ $a^{2}=d$ $,$ $b^{2}=d$ $,c^{2}=d$.$.$ \\ \hline
$J_{4,6}$ $:$ $a^{2}=b\text{ },\text{ }b\circ c=d.$ \\ \hline
$J_{4,7}$ $:$ $a^{2}=b\text{ },\text{ }a\circ b=d\text{ },\text{ }c^{2}=d.$
\\ \hline
$J_{4,8}$ $:$ $a^{2}=c\text{ },\text{ }b^{2}=c\text{ },\text{ }a\circ c=d.$
\\ \hline
$J_{4,9}$ $:a^{2}=c\text{ },\text{ }b^{2}=-c\text{ },\text{ }a\circ c=d$ $,$ 
$b\circ c=d$ $.$ \\ \hline
$J_{4,10}$ $:$ $a^{2}=c\text{ },\text{ }b^{2}=-c\text{ },\text{ }a\circ c=d$ 
$,$ $b\circ c=d$ $,$ $a\circ b=d.$ \\ \hline
$J_{4,11}$ $:$ $a^{2}=b\text{ },\text{ }a\circ b=c$ $,$ $a\circ c=d$ $,$ $%
b^{2}=d.$ \\ \hline
$J_{4,12}$ $:$ $a^{2}=c\text{ },\text{ }a\circ b=d$ $.$ \\ \hline
$J_{4,13}$ $:$ $a^{2}=c\text{ },\text{ }b^{2}=c$ $,$ $a\circ b=d$ $.$ \\ 
\hline
\end{tabular}%
\end{equation*}
\end{theorem}

We see that all nilpotent Jordan algebras of dimension $4$ over algebraic
closed field $K$\ and $ch(K)\neq 2$ are associative except $%
J_{4,6},J_{4,8},J_{4,9}$ and $J_{4,10}.$

\begin{theorem}
Up to isomorphism there exist $9$ commutative nilpotent associative algebras
of dimension $4$ over algebraic closed field $K$\ and $ch(K)\neq 2$ which
are isomorphic to one of the following \ pairwise non-isomorphic commutative
nilpotent associative algebras :%
\begin{equation*}
\begin{tabular}{|l|}
\hline
\textbf{Nilpotent Jordan algebras with Centeral component} \\ \hline
$J_{4,1}=J_{3,1}\oplus J_{1,1}.$ $:$All multiplications are zero. \\ \hline
$J_{4,2}=J_{3,2}\oplus J_{1,1}.$ $:$ $a^{2}=b.$ \\ \hline
$J_{4,3}=J_{3,3}\oplus J_{1,1}.$ $:$ $a^{2}=c$ $,$ $b^{2}=c.$ \\ \hline
$J_{4,4}=J_{3,4}\oplus J_{1,1}.$ $:$ $a^{2}=b$ $,$ $a\circ b=c.$ \\ \hline
\textbf{Nilpotent Jordan algebras without Centeral component} \\ \hline
$J_{4,5}$ $:$ $a^{2}=d$ $,$ $b^{2}=d$ $,c^{2}=d$.$.$ \\ \hline
$J_{4,7}$ $:$ $a^{2}=b\text{ },\text{ }a\circ b=d\text{ },\text{ }c^{2}=d.$
\\ \hline
$J_{4,11}$ $:$ $a^{2}=b\text{ },\text{ }a\circ b=c$ $,$ $a\circ c=d$ $,$ $%
b^{2}=d.$ \\ \hline
$J_{4,12}$ $:$ $a^{2}=c\text{ },\text{ }a\circ b=d$ $.$ \\ \hline
$J_{4,13}$ $:$ $a^{2}=c\text{ },\text{ }b^{2}=c$ $,$ $a\circ b=d$ $.$ \\ 
\hline
\end{tabular}%
\end{equation*}
\end{theorem}

\begin{theorem}
Up to isomorphism there exist $17$ nilpotent Jordan algebras of dimension $4$
over $%
%TCIMACRO{\U{211d} }%
%BeginExpansion
\mathbb{R}
%EndExpansion
$ which are isomorphic to one of the following \ pairwise non-isomorphic
nilpotent Jordan algebras :%
\begin{equation*}
\begin{tabular}{|l|}
\hline
\textbf{Nilpotent Jordan algebras with Centeral component} \\ \hline
$J_{4,1}=J_{3,1}\oplus J_{1,1}.$ $:$All multiplications are zero. \\ \hline
$J_{4,2}=J_{3,2}\oplus J_{1,1}.$ $:$ $a^{2}=b.$ \\ \hline
$J_{4,3}^{\alpha =\pm 1}=J_{3,3}^{\alpha =\pm 1}\oplus J_{1,1}.$ $:$ $%
a^{2}=c $ $,$ $b^{2}=\alpha c$ $.$ $\ \ \ \ \ \ \ \ \left( \alpha =\pm
1\right) $ \\ \hline
$J_{4,4}=J_{3,4}\oplus J_{1,1}.$ $:$ $a^{2}=b$ $,$ $a\circ b=c.$ \\ \hline
\textbf{Nilpotent Jordan algebras without Centeral component} \\ \hline
$J_{4,5}^{\alpha =\pm 1}$ $:$ $a^{2}=d$ $,$ $b^{2}=d$ $,c^{2}=\alpha d$ $.$ $%
\ \ \ \ \ \ \ \ \ \ \ \ \ \ \ \ \ \ \ \ \ \ \ \ \ \ \ \left( \alpha =\pm
1\right) $ \\ \hline
$J_{4,6}$ $:$ $a^{2}=b\text{ },\text{ }b\circ c=d.$ \\ \hline
$J_{4,7}$ $:$ $a^{2}=b\text{ },\text{ }a\circ b=d\text{ },\text{ }c^{2}=d.$
\\ \hline
$J_{4,8}^{\alpha =\pm 1}$ $:$ $a^{2}=c\text{ },\text{ }b^{2}=\alpha c\text{ }%
,\text{ }a\circ c=d.$ $\ \ \ \ \ \ \ \ \ \ \ \ \ \ \ \ \ \ \ \ \ \ \ \
\left( \alpha =\pm 1\right) $ \\ \hline
$J_{4,9}$ $:a^{2}=c\text{ },\text{ }b^{2}=-c\text{ },\text{ }a\circ c=d$ $,$ 
$b\circ c=d$ $.$ \\ \hline
$J_{4,10}$ $:$ $a^{2}=c\text{ },\text{ }b^{2}=-c\text{ },\text{ }a\circ c=d$ 
$,$ $b\circ c=d$ $,$ $a\circ b=d.$ \\ \hline
$J_{4,11}$ $:$ $a^{2}=b\text{ },\text{ }a\circ b=c$ $,$ $a\circ c=d$ $,$ $%
b^{2}=d.$ \\ \hline
$J_{4,12}$ $:$ $a^{2}=c\text{ },\text{ }a\circ b=d$ $.$ \\ \hline
$J_{4,13}^{,\alpha =\pm 1.}$ $:$ $a^{2}=c\text{ },\text{ }b^{2}=\alpha c$ $,$
$a\circ b=d$ $.$ $\ \ \ \ \ \ \ \ \ \ \ \ \ \ \ \ \ \ \ \ \ \ \ \ \ \left(
\alpha =\pm 1\right) $ \\ \hline
\end{tabular}%
\end{equation*}
\end{theorem}

Also, all nilpotent Jordan algebras of dimension $4$ over $%
%TCIMACRO{\U{211d} }%
%BeginExpansion
\mathbb{R}
%EndExpansion
$ are associative except $J_{4,6},J_{4,8}^{\alpha =\pm 1},J_{4,9}$ and $%
J_{4,10}.$

\begin{theorem}
Up to isomorphism there exist $12$ commutative nilpotent associative
algebras of dimension $4$ over $%
%TCIMACRO{\U{211d} }%
%BeginExpansion
\mathbb{R}
%EndExpansion
$ which are isomorphic to one of the following \ pairwise non-isomorphic
commutative nilpotent associative algebras :%
\begin{equation*}
\begin{tabular}{|l|}
\hline
\textbf{Nilpotent Jordan algebras with Centeral component} \\ \hline
$J_{4,1}=J_{3,1}\oplus J_{1,1}.$ $:$All multiplications are zero. \\ \hline
$J_{4,2}=J_{3,2}\oplus J_{1,1}.$ $:$ $a^{2}=b.$ \\ \hline
$J_{4,3}^{\alpha =\pm 1}=J_{3,3}^{\alpha =\pm 1}\oplus J_{1,1}.$ $:$ $%
a^{2}=c $ $,$ $b^{2}=\alpha c$ $.$ $\ \ \ \ \ \ \ \ \left( \alpha =\pm
1\right) $ \\ \hline
$J_{4,4}=J_{3,4}\oplus J_{1,1}.$ $:$ $a^{2}=b$ $,$ $a\circ b=c.$ \\ \hline
\textbf{Nilpotent Jordan algebras without Centeral component} \\ \hline
$J_{4,5}^{\alpha =\pm 1}$ $:$ $a^{2}=d$ $,$ $b^{2}=d$ $,c^{2}=\alpha d$ $.$ $%
\ \ \ \ \ \ \ \ \ \ \ \ \ \ \ \ \ \ \ \ \ \ \ \ \ \ \ \left( \alpha =\pm
1\right) $ \\ \hline
$J_{4,7}$ $:$ $a^{2}=b\text{ },\text{ }a\circ b=d\text{ },\text{ }c^{2}=d.$
\\ \hline
$J_{4,11}$ $:$ $a^{2}=b\text{ },\text{ }a\circ b=c$ $,$ $a\circ c=d$ $,$ $%
b^{2}=d.$ \\ \hline
$J_{4,12}$ $:$ $a^{2}=c\text{ },\text{ }a\circ b=d$ $.$ \\ \hline
$J_{4,13}^{,\alpha =\pm 1.}$ $:$ $a^{2}=c\text{ },\text{ }b^{2}=\alpha c$ $,$
$a\circ b=d$ $.$ $\ \ \ \ \ \ \ \ \ \ \ \ \ \ \ \ \ \ \ \ \ \ \ \ \ \left(
\alpha =\pm 1\right) $ \\ \hline
\end{tabular}%
\end{equation*}
\end{theorem}

\end{document}